\documentclass[]{amsart}
\usepackage[utf8]{inputenc}  
\usepackage{csquotes}
\usepackage[T1]{fontenc}
\usepackage{graphicx}
\usepackage{tikz}
\usepackage{pgfplots}
\usepackage{pgfplotstable}
\pgfplotsset{compat=1.18}
\pgfplotsset{compat=1.18}
\usetikzlibrary{positioning}
\usepackage{float}
\usepackage{algorithm}
\usepackage{algorithmic}
\usepackage[utf8]{inputenc}
\usepackage[export]{adjustbox}
\usepackage{wrapfig}
\usepackage{amsfonts}
\usepackage{amsthm}
\usepackage{array}
\usepackage[top=2cm, bottom=2cm, left=3cm, right=3cm]{geometry}
\usepackage{amssymb}
\usepackage{xcolor}
\usepackage{mathtools}
\usepackage{comment}
\usepackage{algorithm}
\usepackage{algorithmic}
\usepackage{bbm}
\usepackage[main=english]{babel}
\usepackage{subcaption}
\usepackage{dsfont}
\usepackage{bbm}
\usepackage{indentfirst}
\usepackage{tabto}
\usepackage{color}
\usepackage{systeme}
\usepackage{eso-pic}
\usepackage{float}
\usepackage{fancyhdr}
\usepackage{appendix}
\usepackage{listings}
\usepackage{pgf, tikz}
\usepackage[utf8]{inputenc}
\usepackage{lastpage}

\definecolor{bf}{rgb}{0,0,0.6} 
\definecolor{mygray}{gray}{0.85}
\definecolor{darkWhite}{rgb}{0.94,0.94,0.94}

\graphicspath{{images/}}

\newcommand{\espcond}[2]{\mathbb{E}\mathopen{}\left[#1\middle|#2\right]}
\newcommand{\esp}[1]{\mathbb{E}\mathopen{}\left[#1\right]}

\newcommand{\reels}{\mathbb{R}}

\usepackage{thmtools,thm-restate}
\usepackage{stmaryrd}

\newcommand{\FuncDef}[4]{{\left\{\begin{array}{c} #1\longrightarrow #2\\ #3\mapsto #4\\\end{array}\right.}}

\usepackage{hyperref}

\numberwithin{equation}{section}

\usepackage{amsthm, amssymb}

\newtheorem{thm}{Theorem}[section]
\newtheorem{definition}[thm]{Definition}
\newtheorem{prop}[thm]{Proposition}
\newtheorem{lemma}[thm]{Lemma}
\newtheorem{hyp}[thm]{Hypothesis}
\newtheorem{corol}[thm]{Corollary}
\theoremstyle{definition}
\newtheorem{remarque}[thm]{Remark}

\begin{document}
\title{Extragradient methods for mean field games of controls and mean field type FBSDEs}
\author{Meynard Charles}
\begin{abstract}
    In this paper we present a numerical scheme to solve coupled mean field forward-backward stochastic differential equations driven by monotone vector fields. This is based on an adaptation of so called extragradient methods by characterizing solutions as zeros of monotone variational inequalities in a Hilbert space. We first introduce the procedure in the context of mean field games of controls and highlight its connection to the fictitious play. We detail the construction of a sequence of approximate solutions converging to the optimal control associated to the mean field game. Under sufficiently strong monotonicity assumptions, we demonstrate that this sequence converges exponentially fast. Then we extend the method and main results to general forward backward systems of stochastic differential equations that do not necessarily stem from optimal control. 
\end{abstract}
\maketitle
\section{Introduction}
\subsection{General introduction}
In this article, we present a numerical method to approximate the solution of the coupled mean field forward-backward stochastic differential equation (FBSDE)
\begin{equation}
\label{general fbsde introduction}
\left\{
\begin{array}{l}
     \displaystyle X_t=X_0-\int_0^t  F(X_s,U_s,\mathcal{L}(X_s,U_s))ds+\sqrt{2\sigma}B_t, \\
    \displaystyle  U_t= g(X_T,\mathcal{L}(X_T))+\int_t^T G(X_s,U_s,\mathcal{L}\left(X_s,U_s\right))ds-\int_t^T Z_sdB_s, 
\end{array}
\right.
\end{equation}
for a given initial condition $X_0\in L^2(\Omega,\reels^d)$, monotone coefficients $(F,G,g)$, where $\mathcal{L}(X,U)$ indicates the joint law of the random variables $X,U$ and $(Z_t)_{t\in [0,T]}$ is uniquely defined in such a fashion that the process $(U_t)_{t\in [0,T]}$ is progressively adapted with respect to the filtration generated by the Brownian motion $(B_t)_{t\geq 0}$ 

In recent years, there has been considerable interest in mean field forward backward systems \cite{mean-field-FBSDE,bensoussan}. Usually, such systems arise from the study of stochastic optimal control problems with an interacting population, whether it be mean field control \cite{Bensoussan2013MeanField}, or mean field games \cite{Lions-college,LasryLionsMFG} for which this probabilistic formulation consists in an alternative to the study of the master equation \cite{convergence-problem}. In this case, the associated problem can usually be formulated as follows

\begin{equation}
    \label{intro: mfg stochastic system}
\left\{
\begin{array}{l}
     \displaystyle X_t=X_0-\int_0^t  \nabla_p H(X_s,U_s,\mathcal{L}(X_s))ds+\sqrt{2\sigma}B_t, \\
    \displaystyle  U_t= \nabla_x u(T,X_T,\mathcal{L}(X_T))+\int_t^T \nabla_x H(X_s,U_s,\mathcal{L}\left(X_s\right))ds-\int_t^T Z_sdB_s,
\end{array}
\right.
\end{equation}
where $(x,p,\mu)\mapsto H(x,p,m)$ is the Hamiltonian associated to the control problem of players and $u(T,\cdot)$ the terminal value function for a given player. There are however many advantages to the added genrality of considering coefficients $(F,G,g)$ that do not fall into this category even for mean field games. Indeed the class of extended mean field games \cite{extended-mfg} is contained in \eqref{general fbsde introduction}, in particular this includes mean field games of control \cite{propagation-of-monotonicity}. Conceptually, this dependence of the coefficients on the law of the solution makes the study of these systems difficult beyond short horizons of time \cite{mean-field-FBSDE,lipschitz-sol}. The existence and uniqueness of solutions to coupled mean field FBSDEs over arbitrarily long intervals of time remains a challenging problem, and relies most often on monotonicity assumptions. In the particular case of mean field games, a theory of wellposedness has been developed mostly through arguments based on partial differential equations (PDEs) in the flat monotone regime \cite{convergence-problem}. On the other hand, another notion of monotonicity has been used to study forward backward systems directly \cite{Lions-college,bensoussan,monotone-sol-meynard} through the Hilbertian approach. In a series of paper \cite{disp-monotone-1,disp-monotone-2,weaksol-dipmono}, this approach was extended to mean field games under the notion of displacement monotonicity. The method we present in this article relies extensively on this last notion of monotonicity we refer to as $L^2-$monotonicity when coefficients are not gradients.

Part of this article is also dedicated to the numerical approximation of mean field FBSDEs with common noise and the development of approximation schemes for this problem. In mean field games these models arise from the addition of a noise impacting the dynamics of all players in the same fashion \cite{mfg-with-common-noise,Lions-college}. Models in which such noise is purely additive have led to master equations of second order \cite{convergence-problem} and have been studied extensively in the literature on both MFGs \cite{convergence-problem,Bertucci02122023,globalwellposednessdisplacementmonotone,doi:10.1137/21M1450008} and mean field type FBSDEs \cite{probabilistic-mfg,monotone-sol-meynard}. On the other hand common noise can also come from an additional stochastic process, as in \cite{noise-add-variable,common-noise-in-MFG}. This leads us to consider systems of the form 
\begin{equation}
    \label{eq: fbsde with common noise intro}
\left\{
\begin{array}{l}
     \displaystyle X_t=X_0-\int_0^t  F(X_s,U_s,p_s,\mathcal{L}(X_s,U_s|\mathcal{F}^0_s))ds+\sqrt{2\sigma}B_t, \\
    \displaystyle  U_t= U_T+\int_t^T G(X_s,U_s,p_s,\mathcal{L}\left(X_s,U_s|\mathcal{F}^0_s\right))ds-\int_t^T Z_sd(B_s,W_s),\\
    \displaystyle p_t=p_0-\int_0^t b(p_s)ds+\sqrt{2\sigma^0}W_t,\\
    U_T=g(X_T,p_T,\mathcal{L}(X_T|\mathcal{F}^0_T)),
\end{array}
\right.
\end{equation}
where $(\mathcal{F}^0_t)_{t\geq 0}$ represents the filtration of the common noise and  $(W_s)_{s\geq 0}$ is a $\mathcal{F}^0-$Brownian motion. In particular it was proven in \cite{common-noise-in-MFG} that this class of systems includes the case of MFGs with an additive common noise. 

In the specific case of mean field games, many numerical schemes have been proposed to solve the system \eqref{intro: mfg stochastic system}. In the absence of common noise, several methods based on an equivalent PDE formulation have been proposed \cite{achdou-finite-difference-I,achdou-finite-difference-II} and \cite{ziad-kobeissi} for MFGs of controls. Another approach is to rely on the fictitious play algorithm presented in \cite{fictitious-play-cardaliaguet} but this approach does not have any theoritical guarantees whenever coefficients depends on the law of controls, moreover PDE based numerical methods are only reasonable whenever the dimension of the state space is low. On the other hand, probabilistic schemes have also been developed to solve FBSDEs of the form \eqref{general fbsde introduction}. Ranging from iterative methods based on Picard iterations \cite{numerical-FBSDE-I,numerical-FBSDE-II} to cost minimization with machine learning \cite{ML-FBSDE-I}. However either those methods do not have any tractable convergence rate or they are based approximating the decoupling field associated to the FBSDE which leads to costly algorithms, very susceptible to the propagation of errors since the decoupling field must be infered from a finite sample at each time step. Although the method we present in this article is also purely probabilistic, it relies instead on the monotonicity of coefficients and is inspired by the family of extra-gradient algorithms \cite{extra-gradient,Nesterov}. Thanks to this approach we do not need to compute the decoupling field of the FBSDE and are able to solve directly the problem for a single initial condition with an exponential convergence rate in the case of MFG of controls. Moreover, those result still holds in the presence of common noise. In this case, there are very few other numerical methods: PDE methods are usually depreciated since the problem is now infinite dimensional. More generally, this is to the author's knowledge the first scheme able to handle mean field games of controls with common noise and an explicit convergence rate.
\subsection{Main contributions}
In this article, we present a probabilistic numerical method to solve monotone FBSDEs of the form \eqref{general fbsde introduction}. Even outside of the mean field regime, the proposed method is new. For this method, we give theoretical convergence rates. In particular we show that under sufficiently strong monotonicity assumptions linear convergence is achieved by the algorithm including for mean field games of controls. The scheme we propose is not sensitive to the addition of common noise, in the sense that the convergence rate is unchanged (although the simulation cost for each iteration is larger from a practical viewpoint). In a last section we present a particle method inspired by this approach and give explicit bounds on the convergence rate of this numerical scheme. 

\subsection{Organization of the paper}
In Section \ref{section: FBSDE} we start with some reminders on MFGs of controls. Then we expose the numerical scheme proposed in this particular context. We give explicit convergence rates and in particular show that whenever the game is displacement monotone, strongly in the controls, linear convergence is achieved. Then, we extend our results to the general case of \eqref{general fbsde introduction} by taking inspiration from the former case study. In section \ref{section: common noise} we treat the case of mean field type FBSDEs with an independent common noise. Formally the presence of a common noise does not impact monotonicity, consequently the extension is straightforward. In a last short section, we introduce a particle method based on this approach and give explicit convergence rates. Then we present some numerical results illustrating the convergence bounds obtained in the previous sections. 
\subsection{Notation}
\begin{enumerate}
\item[-] let $k\in \mathbb{N}$, $k>0$, for the canonical product on $\reels^k$ we use the notation
\[x\cdot y=\sum_i x_iy_i,\]
and the following notation for the induced norm
\[|x|=\sqrt{x\cdot x}.\]
    \item[-] Let $\mathcal{P}(\reels^d)$ be the set of (Borel) probability measures on $\reels^d$, for $q\geq 0$, we use the usual notation 
\[\mathcal{P}_q(\reels^d)=\left\{\mu\in \mathcal{P}(\reels^d), \quad \int_{\reels^d} |x|^q \mu(dx)<+\infty\right\},\]
for the set of all probability measures with a finite $q$th moment.
\item[-] For two measures $\mu,\nu\in\mathcal{P}(\reels^d)$ we define $\Gamma(\mu,\nu)$ to be the set of all probability measures $\gamma\in\mathcal{
P}(\reels^{2d})$ satisfying 
\[\gamma(A\times \reels^d)=\mu(A) \quad \gamma(\reels^d\times A)=\nu(A),\]
for all Borel set $A$ on $\reels^d$. 
\item[-]The Wasserstein distance between two measures belonging to $\mathcal{P}_q(\reels^d)$ is defined as
\[\mathcal{W}_q(\mu,\nu)=\left(\underset{\gamma\in \Gamma(\mu,\nu)}{\inf}\int_{\mathbb{\reels}^{2d}} |x-y|^q\gamma(dx,dy)\right)^{\frac{1}{q}}.\]
In what follows $\mathcal{P}_q(\reels^d)$ is always endowed with the associated Wasserstein distance, $(\mathcal{P}_q,\mathcal{W}_q)$ being a complete metric space. 
\item[-] We say that a function $U:\mathcal{P}_q(\reels^d)\to \reels^d$ is Lipschitz if
\[\exists C\geq 0, \quad \forall(\mu,\nu)\in\left(\mathcal{P}_q(\reels^d)\right)^2, \quad |U(\mu)-U(\nu)|\leq C\mathcal{W}_q(\mu,\nu).\]
\item[-] Consider $(\Omega, \mathcal{F},\mathbb{P})$ a probability space,
\begin{itemize}
    \item[-] We define 
\[L^q(\Omega, \reels^d)=\left\{ X: \Omega\to \reels^d, \quad \esp{|X|^q}<+\infty\right\}.\]
\item[-] Whenever a random variable $X:\Omega\to \reels^d$ is distributed along $\mu\in \mathcal{P}(\reels^d)$ we use equivalently the notations $\mathcal{L}(X)=\mu$ or  $X\sim\mu$. 
\end{itemize}
\item[-] $\mathcal{M}_{d\times n}(\reels)$ is the set of all matrices of size $d\times n$ with reals coefficients, with the notation $\mathcal{M}_n(\reels)\equiv \mathcal{M}_{n\times n}(\reels)$.
\item[-] $C_b(\reels^d,\reels^k)$ is the set of all continuous bounded functions from $\reels^d$ to $\reels^k$.
\item[-] We say that $U:\mathcal{P}(\reels^d)\to \reels$ is differentiable at $m$ if there exists a continuous map $\psi:\reels^d\to \reels$ such that
\[\forall \nu\in\mathcal{P}(\reels^d) \quad \underset{\varepsilon\to 0}{\lim} \frac{U(m+\varepsilon(\nu-m))-U(m)}{\varepsilon}=\int_{\reels^d}\psi(x)(\nu-m)(dx).\]
In this case we define the derivative of $U$ at $m$, $\frac{\delta U}{\delta m}(m)$ to be one such $\psi$ such that 
\[\int_{\reels^d} \frac{\delta U}{\delta m}(m)(y)m(dy)=0.\]
\item[-] Derivatives of higher order on the space of probability measures are defined by induction and for $k,l,r\in \mathbb{N}, p>0$ and a function 
\[f:\reels^l\times \mathcal{P}_p(\reels^r)\to \reels,\]
we define the notation 
\[\left\|\frac{\delta^k f}{\delta m^k}\right\|_{Lip}=\sup_{\begin{array}{c}
x,x'\in \reels^l\\
m,m'\in \mathcal{P}_p(\reels^r)\\
y,y'\in \reels^{kr}
\end{array}} \frac{|\frac{\delta^k f }{\delta m^k}(x,m,y)-\frac{\delta^k f }{\delta m^k}(x',m',y')|}{|x-x'|+|y-y'|+\mathcal{W}_p(m,m')}.\]
\end{enumerate}
\subsection{Preliminary results}
In this article we fix an horizon $T>0$ and consider a complete probability space $(\Omega,\mathcal{F},\mathbb{P})$ endowed with a $(d+d^0)-$dimensional Brownian motion $(B_t,W_t)_{0\leq t\leq T}$. In particular we remind that the space $L^2(\Omega,\reels)$ endowed with the inner product 
\[\langle \cdot,\cdot \rangle: (X,Y)\mapsto \esp{XY},\]
is a Hilbert space. We use the notation $\mathcal{H}=(L^2(\Omega,\reels),\langle \cdot,\cdot\rangle)$ and the induced norm is denoted by
\[\forall X\in \mathcal{H}, \|X\|=\sqrt{\langle XX\rangle}.\]
We also use the notation $\mathcal{H}^T$ for the Hilbert space of square integrable $\mathcal{F}-$adapted random processes on $[0,T]$
\[\left\{(X_s)_{s\in [0,T]}, \quad \esp{E\int_0^T |X_s|^2ds}<+\infty \text{ and }(X_s)_{s\in [0,T]} \text { is } \mathcal{F}-\text{adapted}\right\},\] 
endowed with the inner product 
\[\langle X,Y\rangle^T=\esp{\int_0^T X_tY_tdt}.\]

\begin{definition}
    For a function $F:\reels^d\times \mathcal{P}_2(\reels^d)\to \reels^d$, we define its lift on $\mathcal{H}^d$, $\tilde{F}$ to be 
    \[\tilde{F}:\left\{
\begin{array}{l}
     \mathcal{H}^d\to \mathcal{H}^d,  \\
     X\mapsto F(X,\mathcal{L}(X)).
\end{array}
\right.
\]
\end{definition}
 We remind the following adaptation of Lemma 2.3 of \cite{disp-monotone-1}
\begin{prop}
\label{prop: expectation to pointwise}
    Let $f:\FuncDef{\reels^{2d}}{\reels}{(x,y)}{f(x,y)}$ be a continuous function,  such that 
    \[\forall (x,y)\in \reels^{2d}, \quad f(x,y)=f(y,x), \quad f(x,x)=0.\]
Suppose that for some $\mu\in\mathcal{P}_2(\reels^d)$ with full support on $\reels^d$ the following holds 
\[\forall (X,Y)\in \mathcal{H}^{2d}, X\sim \mu,Y\sim \mu, \quad \esp{f(X,Y)}\geq 0,\]
then the inequality is satisfied pointwise 
\[\forall (x,y)\in \reels^d,\quad f(x,y)\geq 0.\]
\end{prop}
In particular this direct corollary
\begin{corol}
\label{from lip hilbert to lip x}
Let $F:\reels^d\times \mathcal{P}_2(\reels^d)\to \reels^d$ be a continuous function. Suppose that there exists a constant $C_F$ such that
\[ \forall (X,Y)\in \mathcal{H}^{2d}, \quad \|\tilde{F}(X)-\tilde{F}(Y)\|\leq C_F\|X-Y\|.\]
Then for the same constant $C_F$
\[\forall (\mu,x,y)\in \mathcal{P}_2(\reels^d)\times \left(\reels^d\right)^2, \quad |F(x,\mu)-F(y,\mu)|\leq C_F |x-y|.\]
\end{corol}
\begin{proof}
It suffices to apply the above proposition to 
\[f:(x,y)\mapsto C_F^2|x-y|^2-|F(x,\mu)-F(y,\mu)|^2,\]
for fixed $\mu$ with full support in $\reels^d$. For measure that do not have full support, the result then follows by density and the continuity of $F$ in $\mathcal{P}_2(\reels^d)$ as in \cite{disp-monotone-1}.
\end{proof}

\subsubsection{A primer on monotone variational inequalities in Hilbert spaces}

\quad

Let $(\mathcal{K},\langle \cdot,\cdot \rangle_\mathcal{K})$ indicate a Hilbert space. Consider $v:\mathcal{K}\to \mathcal{K}$ a continuous function and suppose we are interested in the following problem
\begin{equation}
    \label{eq: variational inequality in hilbert}
    \text{find }x^*\in \mathcal{K}, \forall x\in \mathcal{K}, \quad \langle v(x^*),x-x^*\rangle_\mathcal{K}\geq 0.
\end{equation}
Since $\mathcal{K}$ is reflexive, this is clearly equivalent to finding a $x^*$ such that 
\[v(x^*)=0_\mathcal{K}.\]
If we further assume that $v$ is monotone non-degenerate, i.e.
\[\forall x,y\in \mathcal{K}, x\neq y, \langle v(x)-v(y),x-y\rangle_\mathcal{K}>0,\]
then there can exists at most one $x^*$ solving \eqref{eq: variational inequality in hilbert}. In particular if we assume that $v$ is strongly monotone, that is
\[\exists \beta>0, \quad \forall x,y\in \mathcal{K}, \langle v(x)-v(y),x-y\rangle_\mathcal{K}\geq \beta \|x-y\|^2_\mathcal{K},\]
then $v$ is an invertible operator and it follows that there exists a $x^*$ solution to \eqref{eq: variational inequality in hilbert}. A very natural question is how to compute $x^*$ in this setting. To that end let us fix some $x_0\in \mathcal{K}, \gamma>0$ and introduce the following sequence 
\[\forall n\geq 1, \quad x_{n+1}=x_n-\gamma v(x_{n+1}).\]
Letting 
\[v_\gamma :x\mapsto x+\gamma v(x),\]
this is equivalent to 
\[\forall n\geq 1, \quad x_{n+1}=v_\gamma^{-1}(x_n).\]
However since $v$ is $\beta-$strongly monotone, it follows that 
\[\forall x,y\in \mathcal{K}, \quad \langle v_\gamma(x)-v_\gamma(y),x-y\rangle \geq (1+\gamma \beta)\|x-y\|^2_\mathcal{K}.\]
This implies that \[\|v_\gamma^{-1}\|_{Lip}\leq \frac{1}{1+\gamma \beta}.\]
Since $v_\gamma^{-1}$ is a contraction on $\mathcal{K}$, we deduce from the definition of $(x_n)_{n\in \mathbb{N}}$ that it converges exponentialiy to the fixed point of $v_\gamma^{-1}$ which also happens to be $x^*$. By this method we have constructed a sequence $(x_n)_{n\in \mathbb{N}}$ which converges to the unique solution $x^*$ of \eqref{eq: variational inequality in hilbert}. The main issue with the above construction is that the definition of $x_{n+1}$ from $x_n$ is implicit. In fact since it may require a fixed point procedure, computing $x_{n+1}$ at each iteration can be just as hard as computing $x^*$ directly. This is where extra-gradient method come into play, the key idea is to approximate the sequence with 
\[x_{n+1}\approx x_n-\gamma v(x_n-\gamma v(x_n)),\]
at each step. The definition of $x_{n+1}$ is now explicit in function of $x_n$, and it turns out that this approximation yields good results both in theory and practise. In particular, we remind this quite general inequality from \cite{NEURIPS2021_6d65b5ac}
\begin{lemma}
\label{GEG inequality}
Let $(\mathcal{K},\langle \cdot,\cdot \rangle_\mathcal{K})$ be a Hilbert space, and let $(X_n,X_{n+\frac{1}{2}},Y_n)_{n\in \mathbb{N}}\subset\mathcal{K}^3 $ be defined through
\[\left\{\begin{array}{c}
     X_{n+\frac{1}{2}}=X_n-\gamma_n V_n,  \\
     Y_{n+1}=Y_n-V_{n+\frac{1}{2}},\\
     X_{n+1}=\gamma_{n+1}Y_{n+1},
\end{array}\right. \]
for $(V_n,V_{n+\frac{1}{2}})_{n\in \mathbb{N}}\subset \mathcal{K}^2$, $(\gamma_n)_{n\in \mathbb{N}}\subset \reels^+$ a non increasing sequence of step size and 
\[Y_1=\frac{1}{\gamma_1}X_1.\]
Then the following holds for any $x\in \mathcal{K}$
\begin{gather*}2\sum_{i=1}^n \langle V_{i+\frac{1}{2}},X_{i+\frac{1}{2}}-x\rangle_\mathcal{K}\\
    \leq \frac{\|x-X_1\|_\mathcal{K}^2}{\gamma_{1}}+\left(\frac{1}{\gamma_{n+1}}-\frac{1}{\gamma_1}\right)\|x\|^2_\mathcal{K}+\sum_{i=1}^n \left(\gamma_i\| V_{n+\frac{1}{2}}-V_n\|_\mathcal{K}^2-\frac{1}{\gamma_i}\| X_{n+\frac{1}{2}}-X_n\|_\mathcal{K}^2\right)\end{gather*}
\end{lemma}
\begin{proof}
Although the original proof is conducted for a sequence in $\reels^d$, since it only uses classical results on vector spaces endowed with an inner product, there is no difficulty in extending the proof to a general Hilbert space.
\end{proof}
\begin{remarque}
Whenever $\gamma_n\equiv \gamma$ is a constant sequence and $V_n=v(X_n)$ for some function $v$, this is equivalent to 
\[\forall n\geq 1, \quad X_{n+1}=X_n-\gamma v(X_n-\gamma v(X_n)).\]
The above inequality allows to work in a very general framework compared to our very limited introduction on extragradient methods. 
\end{remarque}
\section{Decoupling algorithm}
\label{section: FBSDE}
In this first part, we present the main idea behind the algorithm we introduce to solve FBSDEs numerically. We present it first in the context of mean field games of control, see \cite{propagation-of-monotonicity} for a presentation of mean field games of controls in the displacement monotone setting. Let us emphasize 
that there is already an extensive litterarure on extragradient methods \cite{PMLR,NEURIPS2021_6d65b5ac} however the idea of using them for (possibly mean field type) FBSDEs appears entirely new though it is in our opinion quite natural in the case of MFGs. In this section we also consider dynamics without common noise. 
\subsection{A motivating example from mean field games}
We start with some reminders on mean field games of controls in the displacement monotone framework, from the formulation of the equilibrium condition to the associated FBSDE. In particular, we adopt a strategy which is sligthly different from usual. Instead of introducing an implicit fixed point mapping to write the forward backward system, we show that the equilibrium can be expressed explicitely in function of the coefficients through the use of an infinite dimensional inverse. Although this formulation is strictly equivalent to more standard ones (which in general rely on an implicit fixed point mapping), for mean field games of controls, we believe this makes the exposition clearer and more straightforward (and in particular totally explicit). This idea will also be important in later sections to link MFGs with more general FBSDEs. 
\subsubsection{Uniqueness and characterization of optimal controls}
Let us first recall exactly the formulation of the mean field game problem in this setting. For a given flow of measures $(m_t)_{t\in[0,T]}\in \mathcal{P}_2([0,T],\reels^{2d})$ we consider the cost
\[J(\alpha,(m_t)_{t\in[0,T]})=\left\{\begin{array}{c}
     \displaystyle \esp{g(X^\alpha_T,\mu_T)+\int_0^T L(X^\alpha_t,\alpha_t,m_t)dt},  \\
     \displaystyle X^\alpha_t=X_0-\int_0^t\alpha_tdt+\sqrt{2\sigma} B_t,\\
     \mu_T=\pi_d m_T.
\end{array}\right. ,\]
where $\pi_d m$ indicates the marginal of $m$ over the first $d$ variables for any $m\in \mathcal{P}_2(\reels^{2d})$. We suppose that $J(\alpha,(m_t)_{t\in [0,T]})$ represents the cost paid by a player with control $\alpha$ whenever the distribution of all other players and their controls is $(m_t)_{t\in [0,T]}$. The mean field game problem consists in finding an equilibrium, that is, a distribution $(m^*_t)_{t\in [0,T]}$ such that after solving 
\[\inf_\alpha J\left(\alpha,(m^*_t)_{t\in [0,T]}\right),\]
the players are still be distributed along $(m^*)_{t\in [0,T]}$. There exist several mathematically equivalent rigorous definitions of this problem. In this article we focus on formulating it directly at the level of controls, which is natural for displacement monotone MFGs.
Namely, solving the mean field game problem is equivalent to finding a control $(\alpha^*_t)_{t\in[0,T]}\subset  ^T$ such that 
\begin{equation}
\label{mfgc}
\alpha^*=\underset{\alpha}{\text{arginf}} J\left(\alpha,\left(\mathcal{L}(X^{\alpha^*}_t,\alpha^*_t)\right)_{t\in[0,T]}\right). \end{equation}
Clearly, this definition results from a fixed point of an optimisation problem characteristic of mean field games. At first glance, the existence of such a control, let alone the uniqueness is not clear. Let us start with a standard result on finite dimensional optimal control
\begin{lemma}
\label{lemma: first order optimality conditions}
Suppose that $(x,\alpha)\mapsto (g(x,\mu),L(x,\alpha,m))\in C^{1,1}(\reels^{2d},\reels^2)$ uniformly in $m\in \mathcal{P}_2(\reels^{2d})$
\begin{enumerate}
    \item[-] If $\alpha\in  \left(\mathcal{H}^T\right)^d$ satisfies 
    \begin{equation}
        \label{eq: optimal control problem}
        \alpha\in \underset{\alpha'\in \left(\mathcal{H}^T\right)^d}{\mathrm{arginf}} J(\alpha',(m_t)_{t\in[0,T]}),
    \end{equation}
    then for any $\alpha' \in \left(\mathcal{H}^T\right)^d$ the following holds 
\begin{equation}
\label{first order optimality}
\langle \nabla_x g(X^\alpha_T,\mu_T),X^{\alpha'}_T-X^\alpha_T\rangle+\int_0^T \langle \left(\begin{array}{c}
         \nabla_x L  \\
         \nabla_\alpha L
    \end{array}\right)(X^\alpha_t,\alpha_t,m_t),\left(\begin{array}{c}
         X^{\alpha'}_t-X^\alpha_t  \\
         \alpha'_t-\alpha_t
    \end{array}\right)\rangle dt\geq 0,
    \end{equation}
    where \[\mu_T(dx)=\int_{y\in \reels^d}m_T(dx,dy),\]
    indicates the marginal of $m_T$ over its first $d$ variables.
    \item[-] If, furthermore, $x\mapsto g(x,\mu)$ is convex uniformly in $\mu \in \mathcal{P}_2(\reels^d)$, and $(x,\alpha)\mapsto L(x,\alpha,m)$ is convex, strongly in $\alpha$, uniformly in $m\in \mathcal{P}_2(\reels^{2d})$ then for any flow of measures $(m_t)_{t\in[0,T]}\subset \mathcal{P}_2(\reels^{2d})$, there exists a unique solution to the optimal control problem 
    \[\inf_{\alpha\in \left(\mathcal{H}^T\right)^d} J(\alpha,(m_t)_{t\in[0,T]}),\]
    and it is the unique control satisfying \eqref{first order optimality}. 
\end{enumerate}
\end{lemma}
\begin{proof}
Let us assume that there exists a control $\alpha \in \left(\mathcal{H}^T\right)^d$ satisfying \eqref{eq: optimal control problem}, for a given control $\alpha'\in \left(\mathcal{H}^T\right)^d$, \eqref{first order optimality} is obtained by observing that for any $\lambda\in (0,1]$
\[\frac{1}{\lambda}\left(J(\alpha+\lambda (\alpha'-\alpha),(m_t)_{t\in[0,T]})-J(\alpha,(m_t)_{t\in[0,T]})\right)\geq 0.\]
The result follows naturally by taking the limit as $\lambda \to 0$. 
As for the second proposition, under the additional convexity assumptions
\[\alpha \mapsto J(\alpha,(m_t)_{t\in[0,T]})\]
is strongly convex on $\left(\mathcal{H}^T\right)^d$ and the existence and uniqueness of a minimizer follow from standard optimization theory. Finally let us assume that there exist two controls $\alpha^1,\alpha^2$ satisfying \eqref{first order optimality}, then the convexity of coefficients implies that
\[\|\alpha^1-\alpha^2\|_T\leq 0,\]
yielding $\alpha^1=\alpha^2$ in $\left(\mathcal{H}^T\right)^d$.
\end{proof}
We now make an assumption on the displacement monotonicity of $L$ and $g$.
\begin{hyp}
\label{disp monotone mfg}
\begin{gather*}
\exists c_L>0, \quad \forall (X,Y,\alpha,\alpha')\in \mathcal{H}^{4d}, \quad \\
\langle \nabla_x g(X,\mathcal{L}(X))-\nabla_x g(Y,\mathcal{L}(Y)),X-Y\rangle \geq 0,\\
\langle \left(\begin{array}{c}
\nabla_x L\\
\nabla_\alpha L
\end{array}\right)
(X,\alpha,\mathcal{L}(X,\alpha))-\left(\begin{array}{c}
\nabla_x L\\
\nabla_\alpha L
\end{array}\right)(Y,\alpha',\mathcal{L}(Y,\alpha')),\left(\begin{array}{c}
     X-Y  \\
     \alpha-\alpha'
\end{array}\right)\rangle \geq c_L\|\alpha-\alpha'\|^2
\end{gather*}
\end{hyp}
\begin{remarque}
Hypothesis \ref{disp monotone mfg} implies the convexity of $x\mapsto g(x,\mu)$ and the $c_L-$strong convexity of $(x,\alpha)\mapsto L(x,\alpha,m)$ in $\alpha$ uniformly in $m\in \mathcal{P}_2(\reels^{2d})$, as a direct consequence of Proposition \ref{prop: expectation to pointwise}.
\end{remarque}
We now show that under these assumptions, there exists a unique solution to a FBSDE involving the Hilbertian inverse of $\nabla_\alpha L$. 

\begin{hyp}
    \label{hyp: lipschitz regularity mfgc}
    the derivatives $\nabla_xg,\nabla_xL ,\nabla_\alpha L$ are well defined and continuous in all their arguments. Moreover  \[(x,a,m)\mapsto (\nabla_x g(x,\mu),\nabla_x L(x,a,m),\] is Lipschitz and \[(x,m)\mapsto \nabla_\alpha L(x,a,m),\] is Lipschitz uniformly in $a \in \reels^d$.
\end{hyp}
\begin{lemma}
    \label{lemma: invertibility and wellposedness under disp monotone}
    Under Hypotheses \ref{disp monotone mfg} and \ref{hyp: lipschitz regularity mfgc}, the Hilbertian mapping 
    \[\alpha \mapsto  \nabla_\alpha L(X,\alpha,\mathcal{L}(X,\alpha)),\]
    admits a continuous inverse in $\alpha\in \mathcal{H}^d$ for a given $X\in \mathcal{H}^d$ denoted
    \[\alpha\mapsto \nabla_\alpha \tilde{L}^{-1}(X,\alpha).\]
    Moreover, there exists a Lipschitz continuous function $L_{inv}:\reels^{2d}\times \mathcal{P}_2(\reels^{2d})\to \reels^d$ such that 
    \[\forall (\alpha,X)\in \mathcal{H}^{2d}, \quad \nabla_\alpha \tilde{L}^{-1}(X,\alpha)=L_{inv}(X,\alpha, \mathcal{L}(X,\alpha)),\]
    and for any initial condition $X_0\in \mathcal{H}^d$ there exists a unique strong solution $(X_t,U_t,Z_t)_{t\in[0,T]}$ to the forward backward system

\begin{equation}
\label{MFG FBSDE}
\left\{
\begin{array}{l}
     \displaystyle X_t=X_0-\int_0^t \nabla_\alpha \tilde{L}^{-1}(X_s,U_s)ds+\sqrt{2\sigma}B_t, \\
    \displaystyle  U_t= U_T+\int_t^T \nabla_x\tilde{L}(X_s,\nabla_\alpha \tilde{L}^{-1}(X_s,U_s))ds-\int_t^T Z_sdB_s,\\
    U_T=\nabla_xg(X_T,\mathcal{L}(X_T)),
\end{array}
\right.
\end{equation}
where we remind that for any $(X,U)\in \mathcal{H}^{2d}$
\[ \nabla_x \tilde{L}(X,U)=\nabla_x L(X,U,\mathcal{L}(X,U)).\]
\end{lemma}
\begin{proof}
Let us first observe that $\nabla_\alpha \tilde{L}^{-1}$ is well defined and Lipschitz as a function from $\mathcal{H}^{2d}$ into $\mathcal{H}^d$. For a given $X\in \mathcal{H}^d$, the invertibility follows from the strong monotonicity of $\alpha \mapsto \nabla_\alpha L(X,\alpha)$ on $\mathcal{H}^d$. The fact that the mapping $(X,U)\mapsto \nabla_\alpha \tilde{L}^{-1}(X,U)$ is Lipschitz on $\mathcal{H}^{2d}$ follows from the following inequality
\begin{gather*}
    \forall (X,Y,\alpha,\alpha')\in \mathcal{H}^{4d},\\
    \frac{C_F^2}{2c_L}\|X-Y\|^2+\langle \nabla_\alpha\tilde{L}(X,\alpha)-\nabla_\alpha \tilde{L}(Y,\alpha'),\alpha-\alpha'\rangle \geq \frac{c_L}{2}\|\alpha -\alpha '\|^2,
\end{gather*}
where $C_F$ indicates the Lipschitz norm of $X\mapsto \nabla_\alpha \tilde{L}(X,\alpha)$ which is bounded independently of $\alpha\in \mathcal{H}^d$ by assumption. The fact that $\nabla_\alpha \tilde{L}^{-1}$ can be expressed as the lift of Lipschitz continuous function follows from \cite{monotone-sol-meynard} Lemma 3.10. In fact, one can directly check that for any $(x,u)\in \reels^{2d}, (X,\alpha)\in \mathcal{H}^{2d}$, $L_{inv}$ is given by 
\begin{equation}
    \label{def: Linv}
    L_{inv}(x,u,\mathcal{L}(X,\alpha))=D_p H(x,u,\mathcal{L}(X,\nabla_\alpha \tilde{L}^{-1}(X,\alpha)))\end{equation}

Consequently, \eqref{MFG FBSDE} is a FBSDE with Lipschitz coefficients. It is standard \cite{lipschitz-sol,probabilistic-mfg} that under this condition, wellposedness holds over a sufficiently short horizon of time and we now prove that this system is monotone yielding wellposedness on any time interval. Proceeding with 
\[\forall (X,U)\in \mathcal{H}^{2d} \quad 
\left\{\begin{array}{l}
    \tilde{F}(X,U)=\nabla_\alpha\tilde{L}^{-1}(X,U),\\
     \tilde{G}(X,U)=\nabla_xL\left(X,\nabla_\alpha \tilde{L}^{-1}(X,U),\mathcal{L}\left(X,\nabla_\alpha \tilde{L}^{-1}(X,U)\right)\right).\end{array}\right.\]
By the definition of $F=\nabla_\alpha \tilde{L}^{-1}$ and the displacement monotonicity of $L$
\begin{gather*}
    \forall (X,Y,\alpha,\alpha')\in \mathcal{H}^{4d},\\
    \langle \left(\begin{array}{c}
\tilde{F}\\
\tilde{G}
\end{array}\right)
(X,U)-\left(\begin{array}{c}
\tilde{F}\\
\tilde{G}
\end{array}\right)(Y,V),\left(\begin{array}{c}
     X-Y  \\
     U-V
\end{array}\right)\rangle \geq c_L\|\tilde{F}(X,U)-\tilde{F}(Y,V)\|^2.
\end{gather*}
Since $g$ is also displacement monotone by assumption, we believe the existence and uniqueness of a solution to \eqref{MFG FBSDE} for any $T>0$ follow from standard techniques. However since we couldn't find this exact result in the literature, we briefly explain how to obtain it in Lemma \ref{lemma: wellposedness fbsde strong monotonicity on F} in the Appendix.
\end{proof}
\begin{remarque}
\label{remarque: inverse of lift}
Since we are considering the Hilbertian inverse of \[\alpha \mapsto \nabla_\alpha(X,\alpha, \mathcal{L}(X,\alpha)),\] and not the inverse in a $\alpha \in \reels^d $ for a given measure $m\in \mathcal{P}_2(\reels^{2d})$, it is important to note that we do not have the equality 
\[\nabla_\alpha \tilde{L}^{-1}=D_pH,\]
where $(x,p,m)\mapsto H(x,p,m)$ is the Hamiltonian associated to $L$. In fact $L_{inv}$ is given exactly by \eqref{def: Linv}.
\end{remarque}
We now present how the FBSDE \eqref{MFG FBSDE} relates to equilibria of the mean field game of control \eqref{mfgc}.
\begin{thm}
    \label{thm: existence mfgc}
    Under Hypotheses \ref{disp monotone mfg} and \ref{hyp: lipschitz regularity mfgc}, there exists a unique solution to the mean field game of control \eqref{mfgc}. It is the unique control satisfying 
    \begin{gather}
\label{caracterisation of control mfg}
 \forall \alpha' \in \left(\mathcal{H}^T\right)^d, \quad\\ \nonumber \langle \nabla_x g(X^\alpha_T,\mathcal{L}(X^\alpha_T)),X^{\alpha'}_T-X^\alpha_T\rangle+\int_0^T \langle \left(\begin{array}{c}
         \nabla_x \tilde{L}  \\
         \nabla_\alpha \tilde{L}
    \end{array}\right)(X^\alpha_t,\alpha_t),\left(\begin{array}{c}
         X^{\alpha'}_t-X^\alpha_t  \\
         \alpha'_t-\alpha_t
    \end{array}\right)\rangle dt\geq  0.
    \end{gather}
    and is given by 
    \[\forall t\leq T,\quad \alpha_t=\nabla_\alpha \tilde{L}^{-1}(X_t,U_t),\]
    where $(X_t,U_t,Z_t)_{t\in[0,T]}$ is the unique solution to \eqref{MFG FBSDE}. 
\end{thm}
\begin{proof}
    The fact that any control solution to the mean field game problem must satisfy \eqref{caracterisation of control mfg} is a direct consequence of Lemma \ref{lemma: first order optimality conditions} in the displacement monotone setting. Uniqueness of a solution to the mean field game problem then follows from the fact that, under our displacement monotonicity assumption on $L,g$, there can be at most one control satisfying \eqref{caracterisation of control mfg}. Finally it remains to check that \eqref{caracterisation of control mfg} holds for  
    \[\forall t\leq T,\quad \alpha_t=\nabla_\alpha \tilde{L}^{-1}(X_t,U_t),\]
    which is a direct consequence of the definition of $(X_t,U_t,Z_t)_{t\in[0,T]}$ as a solution to the FBSDE \eqref{MFG FBSDE}.
\end{proof}
Existence and uniqueness results of a solution to the mean field game of control under the standing assumptions is already well known, see for example the recent work \cite{jackson2025quantitativeconvergencedisplacementmonotone}. A novelty here is that the forward backward system can be made explicit by using the Hilbertian inverse instead of relying on an implicit fixed point map, then MFGs of controls can be seen just as a particular case of mean field type FBSDEs. In our opinion, this makes the exposition more straightforward and will also be relevant later in the article. 
\subsubsection{An algorithm from the theory of monotone variational inequalities}
For a given control $(\alpha_t)_{t\in[0,T]}\in \left(\mathcal{H}^T\right)^d$, we first introduce the following backward process 
\[U^\alpha_t=\nabla_xg(X^\alpha_T,\mathcal{L}(X^\alpha_T))+\int_t^T \nabla_xL(X^\alpha_s,\alpha_s,\mathcal{L}(X^\alpha_s,\alpha_s))ds-\int_t^T Z^\alpha_sdB_s.\]
The existence and uniqueness of a pair $(U^\alpha,Z^\alpha)_{t\in[0,T]}$ for a given control is classic in the theory of decoupled forward backward systems, see  \cite{zhang2017backward} Theorem 4.3.1. Indeed since the pair $(\alpha_s,X^\alpha_s)_{s\in [0,T]}$ is known in this case, the above system can be viewed as a standard backward Lipschitz SDE of the form
\begin{equation}
\label{backward sde}
    Y_t=\xi-\int_t^T f(s,\omega,Y_s,Z_s)ds-\int_t^T Z_sdB_s,
\end{equation}
with 
\[f(s,\omega,Y_s,Z_s)=-\nabla_xL(X^\alpha_s(\omega),\alpha_s(\omega),\mathcal{L}(X^\alpha_s,\alpha_s)).\] 
We now introduce the following functional on $\left(\mathcal{H}^T\right)^d$
\[v:\FuncDef{\left(\mathcal{H}^T\right)^d}{\left(\mathcal{H}^T\right)^d}{(\alpha_t)_{t\in [0,T]}}{(\nabla_\alpha L(X^\alpha_t,\alpha_t,\mathcal{L}(X^\alpha_t,\alpha_t))-U^\alpha_t)_{t\in [0,T]}}.\]
Let us insist, as this will be important in later sections on the implementation of our method, that given a control $\alpha$, it is in practise quite straightforward and fast to approximate the couple $(X^\alpha_t,U^\alpha_t)_{t\in [0,T]}$ with a particle method. Indeed the forward component is linear in $\alpha$ and the evolution of the backward component depends only on the couple $(X^\alpha_t,\alpha_t)_{t\in [0,T]}$ which are then known. In fact it is well known that in this case
\[U^\alpha_t=\espcond{\nabla_x g(X^\alpha_T,\mathcal{L}(X^\alpha_T))+\int_t^T \nabla_xL(X^\alpha_s,\alpha_s,\mathcal{L}(X^\alpha_s,\alpha_s))ds}{\mathcal{F}_t},\]
and it can be approximated efficiently with regresssion based methods.
\begin{lemma}
    \label{lemma: def v mfgc}
Under Hypothesis \ref{hyp: lipschitz regularity mfgc}, there exists a constant $C_v$ depending only on the Lipschitz constants of $\nabla_xg,\nabla_x L$ and $T$ only such that 
\[\forall (\alpha,\alpha')\in \left(\mathcal{H}^T\right)^{2d},\quad \|v(\alpha)-v(\alpha')\|_T\leq C_v \|\alpha-\alpha'\|_T.\]
Furthermore, if Hypothesis \ref{disp monotone mfg} holds then $v$ is $c_L$-strongly monotone on $\left(\mathcal{H}^T\right)^d$ and a control $\alpha^*\in \left(\mathcal{H}^T\right)^d$ is a solution to the mean field game of control problem \eqref{mfgc} if and only if
\[\forall \alpha\in \left(\mathcal{H}^T\right)^d \quad \langle v(\alpha^*),\alpha-\alpha^*\rangle \geq 0.\]
\end{lemma}
\begin{proof}
Let us first prove the statement on the Lipschitz regularity of $v$. It is quite evident that for any two controls $(\alpha,\alpha')\in \left(\mathcal{H}^T\right)^{2d}$,
\[\int_0^T \|X^\alpha_t-X^{\alpha'}_t \|dt\leq \|\alpha-\alpha'\|_T.\]
Then applying Ito's Lemma for fixed $t\in [0,T]$ 
\begin{align*}
\| U^\alpha_T-U^{\alpha'}_T\|^2=&\| U^\alpha_t-U^{\alpha'}_t\|^2+\int_t^T \| Z^\alpha_s-Z^{\alpha'}_s\|^2ds\\
&-\esp{\int_t^T \left(\nabla_x \tilde{L}(X^\alpha_s,\alpha_s)-\nabla_x \tilde{L}(X^{\alpha'}_s,{\alpha'}_s)\right)\cdot \left(U^\alpha_s-U^{\alpha'}_s\right)ds}
\end{align*}
which implies that for any $t\in [0,T]$
\[ \|U^\alpha_t-U^{\alpha'}_t\|^2\leq (\|\nabla_xg\|_{Lip}+T\|\nabla_xL\|_{Lip})\|\alpha-\alpha'\|^2_T+\frac{\|\nabla_xL\|_{Lip}}{4}\int_t^T \| U^\alpha_s-U^{\alpha'}_s\|^2ds,\]
and the conclusion follows from Gronwall lemma. For the second part of the lemma, by definition of $(U^\alpha_t)_{t\in [0,T]}$, 
\begin{gather*}\forall \alpha,\alpha' \in \left(\mathcal{H}^T\right)^d, \\
    \langle v(\alpha),\alpha'-\alpha\rangle^T\\
    =\langle \nabla_x g(X^\alpha_T,\mathcal{L}(X^\alpha_T)),X^{\alpha'}_T-X^\alpha_T\rangle+\int_0^T \langle \left(\begin{array}{c}
         \nabla_x \tilde{L}  \\
         \nabla_\alpha \tilde{L}
    \end{array}\right)(X^\alpha_t,\alpha_t),\left(\begin{array}{c}
         X^{\alpha'}_t-X^\alpha_t  \\
         \alpha'_t-\alpha_t
    \end{array}\right)\rangle dt,
\end{gather*}
and the result follows from Hypothesis \ref{disp monotone mfg} and Theorem \ref{thm: existence mfgc}
\end{proof}
With this, we have proved that the problem of solving for the mean field game equilibrium can be reformulated as solving a variational inequality on the space of controls $\left(\mathcal{H}^T\right)^d$.
For Lipschitz functions $v$ on the space of controls $\left(\mathcal{H}^T\right)^d$, we introduce the standard notation 
\[\|v\|_{Lip}=\sup_{\alpha,\alpha'\in \left(\mathcal{H}^T\right)^d}\frac{\|v(\alpha)-v(\alpha')\|_T}{\|\alpha-\alpha'\|_T}.\]
Taking inspiration from the dual extrapolation algorithm introduced in \cite{extra-gradient}, we present an algorithm that converges to the mean field game optimal control starting from any initial condition
\begin{thm}
\label{thm: convergence mfgc}
Take an initial condition $\alpha_1\in \left(\mathcal{H}^T\right)^{d}$, for $n\geq 1$ we introduce the sequences 
\[\left\{\begin{array}{c}
     \alpha_{n+\frac{1}{2}}=\alpha_n-\gamma v(\alpha_n),  \\
     \alpha_{n+1}=\alpha_n-\gamma v(\alpha_{n+\frac{1}{2}}).
\end{array}\right. \]
If 
\[\gamma\leq \frac{1}{\|v\|_{Lip}},\]
then under Hypotheses \ref{hyp: lipschitz regularity mfgc} and \ref{disp monotone mfg}, letting 
\[\bar{\alpha}_n=\frac{1}{n}\sum_{i=1}^n \alpha_{i+\frac{1}{2}},\]
and $\alpha^*$ be the unique control solution of the MFG of control \eqref{mfgc}, the following holds 
\[\forall n\geq 1,\quad \|\alpha^*-\bar{\alpha}_n\|_T^2\leq \frac{\| \alpha^*-\alpha_1\|^2_T}{2\gamma c_L n}.\]
\end{thm}
\begin{proof}
Let us first observe by an application of Lemma \ref{GEG inequality},for any control $\alpha \in \left(\mathcal{H}^T\right)^d$, 
\[\forall n\geq 1,\quad \sum_{i=1}^n \langle v(\alpha_{i+\frac{1}{2}}),\alpha_{i+\frac{1}{2}}-\alpha \rangle^T\leq \frac{\|\alpha-\alpha_1\|^2_T}{2\gamma}.\]
On the other hand, introducing the notation 
\[\bar{\alpha}_n=\frac{1}{n}\sum_{i=1}^n \alpha_{i+\frac{1}{2}},\]
\begin{align*}
\sum_{i=1}^n \langle v(\alpha_{i+\frac{1}{2}}),\alpha_{i+\frac{1}{2}}-\alpha \rangle^T&= \sum_{i=1}^n \langle v(\alpha_{i+\frac{1}{2}})-v(\alpha),\alpha_{i+\frac{1}{2}}-\alpha\rangle^T+\langle v(\alpha),\alpha_{i+\frac{1}{2}}-\alpha \rangle^T\\
&\geq n\langle v(\alpha),\bar{\alpha}_n-\alpha\rangle^T+c_L\sum_{i=1}^n \|\alpha-\alpha_{i+\frac{1}{2}}\|^2_T \\
&\geq n\langle v(\alpha),\bar{\alpha}_n-\alpha\rangle^T+n c_L\| \bar{\alpha}_n-\alpha\|^2_T,
\end{align*}
where we used the strong monotonicity of $v$ on $\left(\mathcal{H}^T\right)^d $.
Putting together those two inequalities and evaluating for $\alpha=\alpha^*$ solution of the mean field game problem yields 
\[\forall n\geq 1, \quad \|\alpha^*-\bar{\alpha}_n\|_T^2\leq \frac{\| \alpha^*-\alpha_1\|_T^2}{2\gamma c_L n}.\]
\end{proof}
\begin{remarque}
A corollary of this proof is that 
\[\sum_{n=1}^{+\infty} \|\alpha^*-\alpha_{n+\frac{1}{2}}\|^2_T<+\infty,\]
which implies that eventually, the sequence of last iterates $(\alpha_{n+\frac{1}{2}})_{n\in \mathbb{N}}$ converges to $\alpha^*$ with
\[\forall C>0, \quad \exists n_0, \forall n\geq n_0,\quad \| \alpha_{n+\frac{1}{2}}-\alpha^*\|_T^2\leq \frac{C}{n},\]
but this does not give an explicit rate, hence the choice of working rather on $(\bar{\alpha}_n)_{n\in \mathbb{N}}$
\end{remarque}
The above iterative procedure can be generalized to other monotonicity conditions, as we will see later. In the specific setting of strong monotonicity in the control,  we can show that for asufficiently small step size, the convergence of the last iterate is exponential 
\begin{thm}
    \label{lemma: exponential convergence}
Under the same assumptions as Theorem \ref{thm: convergence mfgc} and with the same notations, if 
\[\gamma< \min \left(\frac{1}{2\|v\|_{Lip}},\frac{c_L}{\|v\|_{Lip}^2}\right),\]
then there exists a constant $\lambda\in (0,1)$ depending only on $\gamma, \|v\|_{Lip}$ and $c_L$ such that 
\[\forall n\in \mathbb{N}, \quad \|\alpha_n-\alpha^*\|_T\leq \lambda^n \| \alpha_0-\alpha^*\|_T.\]
\end{thm}
\begin{proof}
In the constant step regime let us first remark that the sequence $(\alpha_i)_{i\in \mathbb{N}}\subset \left(\mathcal{H}^T\right)^d$ is defined through
\[\alpha_{i+1}=\alpha_i-\gamma v(\alpha_i-\gamma v(\alpha_i)).\]
Letting $\alpha^*$ be the control associated to \eqref{FBSDE}, it follows that 
\begin{align*}
    \|\alpha_{i+1}-\alpha^*\|_T^2&=\|\alpha_i-\alpha^*-\gamma v(\alpha_i-\gamma v(\alpha_i))\|_T^2\\
    &=\|\alpha_i-\alpha^*\|_T^2+\gamma^2 \|v(\alpha_i-\gamma v(\alpha_i))\|_T^2-2\gamma \langle \alpha_i-\alpha^*, v(\alpha_i-\gamma v(\alpha_i)\rangle^T.
\end{align*}
Using the fact that 
\[ v(\alpha^*)=0,\]
we get 
\begin{align*}\|\alpha_{i+1}-\alpha^*\|_T^2=&\|\alpha_i-\alpha^*\|_T^2+\gamma^2 \|v(\alpha_i-\gamma v(\alpha_i))\|_T^2\\&-2\gamma \langle \alpha_i-\alpha^*, v(\alpha_i-\gamma v(\alpha_i))-v(\alpha^*)\rangle^T.\end{align*}
Using the strong monotonicity of $v$, 
\begin{align*}\|\alpha_{i+1}-\alpha^*\|_T^2&\leq \|\alpha_i-\alpha^*\|_T^2+\gamma^2 \|v(\alpha_i-\gamma v(\alpha_i))\|_T^2-2\gamma c_L \|\alpha_i-\gamma v(\alpha_i)-\alpha^*\|_T^2\\
    & -2\gamma^2 \langle v(\alpha_i), v(\alpha_i-\gamma v(\alpha_i))-v(\alpha^*)\rangle^T\\
\end{align*}
Which implies that 
\begin{align*}
\|\alpha_{i+1}-\alpha^*\|_T^2\leq&(1-2\gamma c_L) \|\alpha_i-\alpha^*\|_T^2 -2\gamma^2 \langle v(\alpha_i)-v(\alpha_i-\gamma v(\alpha_i)), v(\alpha_i-\gamma v(\alpha_i))\rangle^T\\
& +4\gamma^2c_L \langle \alpha_i-\alpha^*,v(\alpha_i)\rangle^T-\gamma^2 \|v(\alpha_i-\gamma v(\alpha_i))\|_T^2-2\gamma^3 c_L \| v(\alpha_i)\|_T^2.
\end{align*}
Since $v$ is Lipschitz clearly
\begin{align*}\|\alpha_{i+1}-\alpha^*\|_T^2\leq&(1-2\gamma c_L+4\gamma^2 c_L\|v\|_{Lip})  \|\alpha_i-\alpha^*\|_T^2-\gamma^2 \|v(\alpha_i-\gamma v(\alpha_i))\|_T^2\\
    &-2\gamma^3 c_L \| v(\alpha_i)\|_T^2+2\gamma^3\|v\|_{Lip}\|v(\alpha_i)\|_T\|v(\alpha_i-\gamma v(\alpha_i))\|_T,\end{align*}
and we finally obtain 
\begin{align*}\|\alpha_{i+1}-\alpha^*\|_T^2\leq&(1-2\gamma c_L+4\gamma^2 c_L\|v\|_{Lip})  \|\alpha_i-\alpha^*\|_T^2\\&+\gamma^2\left(\gamma \frac{\|v\|_{Lip}^2}{c_L}-1\right) \|v(\alpha_i-\gamma v(\alpha_i))\|_T^2.\end{align*}
In particular, letting 
\[\gamma^*=\min \left(\frac{1}{2\|v\|_{Lip}},\frac{c_L}{\|v\|_{Lip}^2}\right),\]
for any $\gamma<\gamma^*$ the sequence $(\alpha_i)_{i\in \mathbb{N}}$ converges exponentially at the following rate 
\[\forall n\in \mathbb{N} \quad \|\alpha_n-\alpha^*\|^2_T \leq (1-2\gamma c_L+4\gamma^2c_L\|v\|_{Lip}^2 )^n\|\alpha_0-\alpha^*\|_T^2.\]
\end{proof}
\begin{remarque}
    \label{remarque: exponential convergence only when strong monotonicity in alpha}
   It is well known that another regime of wellposedness for the mean field game problem \ref{mfgc} is a strong monotonicity condition on the forward variable $X$. Namely, if coefficients are Lipschitz, $\nabla_\alpha \tilde{L}$ has a well defined, Lipschitz inverse and the following monotonicity condition is satisfied
\begin{gather*}
\exists c_L>0, \quad \forall (X,Y,\alpha,\alpha')\in \mathcal{H}^{4d}, \quad \\
\langle \nabla_x g(X,\mathcal{L}(X))-\nabla_x g(Y,\mathcal{L}(Y)),X-Y\rangle \geq c_L \|X-Y\|^2,\\
\langle \left(\begin{array}{c}
\nabla_x L\\
\nabla_\alpha L
\end{array}\right)
(X,\alpha,\mathcal{L}(X,\alpha))-\left(\begin{array}{c}
\nabla_x L\\
\nabla_\alpha L
\end{array}\right)(Y,\alpha',\mathcal{L}(Y,\alpha')),\left(\begin{array}{c}
     X-Y  \\
     \alpha-\alpha'
\end{array}\right)\rangle \geq c_L\|X-Y\|^2
\end{gather*}
then there exists a unique strong solution to \eqref{MFG FBSDE}(see \cite{common-noise-in-MFG} for example). Although most results presented in this article extend easily to this setting (in particular the convergence of $(\bar{\alpha}_n)_{n\in \mathbb{N}}$),  it is unclear to us whether this Theorem on the exponential convergence of last iterates remains true. This is why we presented first a more general result on the convergence of averaged controls $(\bar{\alpha}_n)_{n\in \mathbb{N}}$. 
\end{remarque}
\subsubsection{Link with the fictitious play}
The algorithm we present is clearly different from the fictitious play algorithm in mean field game. The most obvious difference being that we consider dynamic on controls rather than on the distribution of players, but the learning procedure is also quite different. 
Nevertheless, there is a strong link between the two algorithms, namely the fictitious play can also be interpreted as solving a variational inequality. 

In this section letting $(\mu_t)_{t\in[0,T]}\subset \mathcal{P}_2(\reels^d)$ we consider the cost
\[J(\alpha,\mu_t)_{t\in[0,T]})=\left\{\begin{array}{c}
     \displaystyle \esp{g(X^\alpha_T,\mu_T)+\int_0^T \mathbf{L}(X^\alpha_t,\alpha_t,\mu_t)dt}  \\
     \displaystyle X^\alpha_t=X_0-\int_0^t\alpha_tdt+\sqrt{2\sigma} B_t\quad 
\end{array}\right.\]
with 
\[\mathbf{L}(x,\alpha,\mu)=L(x,\alpha)+f(x,\mu),\]
and the mean field game problem of finding a flow of measure $(\mu^*)_{t\in [0,T]}$ such that 
\[\left\{ \begin{array}{c}
    \alpha^*=\underset{\alpha}{\text{arginf}} J(\alpha,(\mu^*)_{t\in [0,T]}),\\
    \forall t\in [0,T], \quad \mu^*_t=\mathcal{L}(X^{\alpha^*}_t).
\end{array}\right.\]
Let
\[P_a=\left\{ (m_t)_{t\in [0,T]}\subset \mathcal{P}_2(\reels^{2d}), \quad \exists \alpha \in \left(\mathcal{H}^T\right)^d, m=\left(\mathcal{L}(X^\alpha_t,\alpha_t)\right)_{t\in [0,T]}\right\},\]
the set of all possible distributions. Indicating the duality product between $f\in C_b(\reels^d)$ and $\mu\in \mathcal{P}(\reels^d)$ by
\[\left( f,\mu\right)=\int_{\reels^d} f(x)\mu(dx),\]
we introduce the following duality product between flow of measures $(m_t)_{t\in [0,T]}\in C([0,T],\mathcal{P}_2(\reels^{2d}))$ and functions $(f,g)\in C([0,T],C(\reels^2,\reels))\times C(\reels^{2d},\reels)$
\[((f,g),(m_t)_{t\in [0,T]})_T= (g,m_T)+\int_0^T (f,m_t)dt.\]
For a given  $(m_t)_{t\in [0,T]}\in P_a$ with marginal $(\mu_t)_{t\in [0,T]}$ over the first $d$ variables we define
\[v(m)=\left((t,x,\alpha)\mapsto \mathbf{L}(x,\alpha,\mu_t),x\mapsto g(x,\mu_T)\right).\]
By definition of $P_a$, it follows that for any $m'\in P_a$ there exists a control $\alpha^{m'}\in \left(\mathcal{H}^T\right)^d$ such that
\[\forall t\in [0,T]\quad  m_t=\mathcal{L}(X^{\alpha^{m'}}_t,\alpha^{m'}_t),\]
moreover the following holds
\[(v(m),m')_T=J((\alpha^{m'}_t)_{t\in [0,T]},(\mu_t)_{t\in [0,T]}).\]
In particular solving the mean field game problem is equivalent to finding $m^*\in P_a$ such that 
\[\forall m\in P_a,\quad ( v(m^*),m-m^*)_T\geq 0,\]
Moreover let us observe that for any $m,n\in P_a$, with marginal over the first $d$ variables being respectively $\mu^m,\mu^n$ 
\begin{gather*}( v(m)-v(n),m-n)_T\\=(g(\cdot,\mu^m_T)-g(\cdot,\mu^n_T),\mu^m_T-\mu^n_T)+\int_0^T(f(\cdot,\mu^m_t)-f(\cdot,\mu^n_t),\mu^m_t-\mu^n_t) dt.\end{gather*}
Consequently, if both $f$ and $g$ are flat monotone then $v$ is a monotone functional from $P_a$ into $C([0,T],C(\reels^2,\reels))\times C(\reels^{2d},\reels)$ with the duality product $(\cdot,\cdot)_T$. 
\subsection{Beyond mean field games: monotone FBSDE}
In the case of mean field game, there is a canonical monotone functional arising from the optimal control problem behind the game. We are now interested in extending the algorithm presented to general monotone FBSDEs. 
We consider the following system 
\begin{equation}
\label{FBSDE}
\left\{
\begin{array}{l}
     \displaystyle X_t=X_0-\int_0^t  F(X_s,U_s,\mathcal{L}(X_s,U_s))ds+\sqrt{2\sigma}B_t, \\
    \displaystyle  U_t= g(X_T,\mathcal{L}(X_T))+\int_t^T G(X_s,U_s,\mathcal{L}\left(X_s,U_s\right))ds-\int_t^T Z_sdB_s.
\end{array}
\right.
\end{equation}
Clearly, the forward backward system associated to the mean field game of control falls into a particular case of \eqref{FBSDE}.
We place ourselves under the following assumptions 
\begin{hyp}
\label{hyp: lipschitz fbsde}
$F,G:\reels^{2d}\times \mathcal{P}_2(\reels^{2d})\to \reels^d$ are Lipschitz and $g:\reels^d\times \mathcal{P}_2(\reels^d)\to \reels^d$ is Lipschitz
\end{hyp}
We also assume that coefficients are $L^2-$monotone,
\begin{hyp}
\label{hyp: L2 monotonicity}
\begin{gather*}
\exists c_F>0, \quad \forall (X,Y,U,V)\in \mathcal{H}^{4d}, \quad \\
\langle  g(X,\mathcal{L}(X))-g(Y,\mathcal{L}(Y)),X-Y\rangle \geq 0,\\
\langle \left(\begin{array}{c}
F\\
G
\end{array}\right)
(X,Y,\mathcal{L}(X,Y))-\left(\begin{array}{c}
F\\
G
\end{array}\right)(Y,V,\mathcal{L}(Y,V)),\left(\begin{array}{c}
     X-Y  \\
     U-V
\end{array}\right)\rangle \geq c_F\|U-V\|^2
\end{gather*}
\end{hyp}
Obviously, this assumption corresponds to the displacement monotonicity assumption on the Lagrangian $L$ in the case of MFGs. It has already been observed numerous times that the wellposedness of this system under $L^2-$monotonicity does not depend on whether or not the coefficients $F,G$ are gradients. First in Lions' Lectures with the Hilbertian approach to master equations \cite{Lions-college,lipschitz-sol}, then in \cite{bensoussan} with a purely probabilistic approach, and more recently in \cite{monotone-sol-meynard} for weak solutions. 
\begin{thm}
    \label{thm: existence fbsde}
    Under Hypotheses \ref{hyp: lipschitz fbsde} and \ref{hyp: L2 monotonicity}, for any admissible initial condition $X_0\in \mathcal{H}^d$ there exists a unique strong solution $(X_t,U_t,Z_t)_{t\in [0,T]}$ to the FBSDE \eqref{FBSDE}.
\end{thm}
\begin{proof}
The existence of a strong solution is a direct consequence of \cite{monotone-sol-meynard} Lemma 3.21 combined with Lemma 3.9. As for uniqueness it is a natural consequence of the strong monotonicity of coefficients. Let us fix an initial condition $X\in \mathcal{H}^d$ an assume that there exist two solutions $(X^i_t,U^i_t,Z^i_t)_{t\in [0,T]}$ for $i=1,2$ starting from $X_0^i=X$. Defining
\[\forall t\in [0,T], \quad V_t=(U^1_t-U^2_t)\cdot (X^1_t-X^2_t), \]
using the monotocity of $(F,G),g$ and applying Ito's Lemma to $(V_t)_{t\in [0,T]}$ we get
\[\forall t\in [0,T], \quad 0\leq \esp{V_T}\leq \esp{V_t}\leq \esp{V_0}=0.\]
Then using this time the strong monotonicity of $(F,G)$ in $U$ it follows that 
\[\esp{\int_0^T|U^1_t-U^2_t|^2dt}=0,\]
and by a direct application of Gronwall's lemma we also have
\[\forall t\in [0,T],\quad \|X^1_t-X^2_t\|=0,\]
hence the uniqueness of a strong solution.
\end{proof}
A natural question at this point is whether we can associate a monotone variational inequality to this system as we did in the case of mean field games \ref{mfgc}. To that end, we introduce the following parametrization 
\[\forall \alpha \in \left(\mathcal{H}^T\right)^d, \quad X^\alpha_t=X_0-\int_0^t \alpha_s ds+\sqrt{2\sigma}B_t.\]
In the case of MFGs, this comes from the formulation of the control problem. Here, it can be seen as a linearization of the FBSDE \eqref{FBSDE}. In particular, $(X^{\alpha^*}_t)_{t\in [0,T]}\overset{\mathcal{H}^T}{=}(X_t)_{t\in [0,T]}$ if and only if
\begin{equation}
\label{eq: relationship between control and u fbsde}
\esp{\int_0^T|\alpha^*_t-F(X_t,U_t,\mathcal{L}(X_t,U_t))|^2dt}=0.
\end{equation} 
Let us now remark the following
\begin{lemma}
    \label{lemma: property v fbsde}
 under Hypotheses \ref{hyp: lipschitz fbsde} and \ref{hyp: L2 monotonicity}, the Hilbertian lift of $F$, $\tilde{F}$ 
 \[U\mapsto \tilde{F}(X,U),\]
 has a well defined inverse in $U\in \mathcal{H}^d$ for any given $X\in \mathcal{H}^d$. Moreover, denoting this inverse by $\tilde{F}_u^{-1}$, there exists a constant $C$ depending only on $\|\tilde{F}\|_{Lip}$ and $c_F$ such that
\[ \forall (X,Y,\alpha,\alpha')\in \mathcal{H}^{4d},\quad \| \tilde{F}^{-1}_u (X,\alpha)-\tilde{F}^{-1}_u (Y,\alpha')\|\leq C\left(\|X-Y\|+\|\alpha-\alpha'\|\right).\]
\end{lemma}
\begin{proof}
    Both the invertibility and the regularity of the inverse follows naturally from the strong monotonicity of $\tilde{F}$ on $\mathcal{H}^d$. This is a straighforward adaptation of Lemma \ref{lemma: invertibility and wellposedness under disp monotone}.
\end{proof}
Since it has now been established that $\tilde{F}$ is invertible, \eqref{eq: relationship between control and u fbsde} implies that a control $\alpha$ solves the FBSDE \eqref{FBSDE} if and only if 
\[(U_t)_{t\in [0,T]}=(\tilde{F}^{-1}_u(X^\alpha_t,\alpha_t))_{t\in [0,T]}.\]
In light of this new representation, we introduce the following linear form
\begin{lemma}
For a given $\alpha \in \left(\mathcal{H}^T\right)^d$, let  $l(\alpha)$ be the bilinear form defined by 
\begin{gather*}\forall \alpha',\alpha''\in \left(\mathcal{H}^T\right)^d,\quad l(\alpha)(\alpha',\alpha'')=\\
\langle g(X^\alpha_T),X^{\alpha'}_T-X^{\alpha''}_T\rangle\\+\int_0^T\langle  \tilde{F}^{-1}_u(X^\alpha_t,\alpha_t),\alpha'_t-\alpha''_t\rangle+\langle \tilde{G}(X^\alpha_t,\tilde{F}^{-1}_u(X^\alpha_t,\alpha_t)),X^{\alpha'}_t-X^{\alpha''}_t\rangle dt.
\end{gather*}
Under Hypotheses \ref{hyp: L2 monotonicity} and \ref{hyp: lipschitz fbsde}
\begin{enumerate} 
    \item[-]$\forall (\alpha,\alpha')\in \mathcal{H}^T$ \[(l(\alpha)-l(\alpha'))(\alpha,\alpha')\geq c_F\int_0^T\|\tilde{F}^{-1}_u (X^\alpha_t,\alpha_t)-\tilde{F}^{-1}_u (X^{\alpha'}_t,\alpha'_t)\|^2dt\]
    \item[-] Let 
    \[\forall t\in [0,T] \quad  \alpha^*_t= \tilde{F}(X_t,U_t),\] 
for $(X_t,U_t,Z_t)_{t\in [0,T]}$ the unique solution to \eqref{FBSDE}, then for any $(\alpha',\alpha'')\in (\mathcal{H}^T)^{2d}$
\[l(\alpha^*)(\alpha',\alpha'')=0.\]
\end{enumerate}
\end{lemma}
\begin{proof}
This first statement is a direct consequence of the monotonicity of coefficients. Following Hypothesis \ref{hyp: L2 monotonicity}, we already know that for all $(X,Y,U,V)\in \mathcal{H}^{4d}$, 
\[\langle \tilde{F}(X,U)-\tilde{F}(Y,V),U-V\rangle+ \langle \tilde{G}(X,U)-\tilde{G}(Y,V),X-Y\rangle\geq c_F\|U-V\|^2.\]
By choosing $U=\tilde{F}^{-1}_u(X,\alpha), V=\tilde{F}^{-1}_u(Y,\alpha')$ for some $\alpha,\alpha'\in \mathcal{H}^d$, we get 
\begin{gather*}\langle \alpha-\alpha',\tilde{F}^{-1}_u(X,\alpha)-\tilde{F}^{-1}_u(Y,\alpha)\rangle+ \langle \tilde{G}(X,\tilde{F}^{-1}_u(X,\alpha))-\tilde{G}(Y,\tilde{F}^{-1}_u(Y,\alpha)),X-Y\rangle\\
    \geq c_F\|\tilde{F}^{-1}_u(X,\alpha)-\tilde{F}^{-1}_u(Y,\alpha)\|^2,\end{gather*}
and the first statetement follows naturally. 

As for the second statement, it suffices to remark that 
\[\int_0^T\|X^{\alpha^*}_t-X_t\|^2dt=0.\]
\end{proof}
We now have all the tools at hand to present a monotone functional associated to the FBSDE \eqref{FBSDE}
\begin{lemma}
\label{lemma: def v fbsde}
Under Hypotheses \ref{hyp: L2 monotonicity} and \ref{hyp: lipschitz fbsde}, for any $\alpha \in \left(\mathcal{H}^T\right)^d, X_0\in \mathcal{H}^d$ the system 
\begin{equation}
\label{parametrized fbsde}
\left\{\begin{array}{c}
\displaystyle X^\alpha_t=X_0-\int_0^t \alpha_s ds +\sqrt{2\sigma} B_t,\\
\displaystyle U^\alpha_t=\tilde{g}(X^\alpha_T)+\int_t^T \tilde{G}(X^\alpha_s,\tilde{F}^{-1}_u(X^\alpha_s,\alpha_s))ds-\int_t^T Z^\alpha_s dB_s,
\end{array}\right.
\end{equation}
has a unique solution $(X^\alpha_t,U^\alpha_t,Z^\alpha_t)_{t\in [0,T]}$. Moreover, letting
\[v:\FuncDef{\left(\mathcal{H}^T\right)^d}{\left(\mathcal{H}^T\right)^d}{(\alpha_t)_{t\in [0,T]}}{ \left(\tilde{F}^{-1}_u(X^\alpha_t,\alpha_t)-U^\alpha_t\right)_{t\in [0,T]}},\]
\begin{enumerate}
\item[-] there exists a constant $C$ depending only on $T,\|(\tilde{F},\tilde{G})\|_{Lip},\|\tilde{g}\|_{Lip}$ and $c_F$ such that 
\[ \forall (\alpha,\alpha')\in (\mathcal{H}^T)^{2d},\quad \| v(\alpha)-v(\alpha')\|_T\leq C \|\alpha-\alpha'\|_T.\]
\item[-] $\forall (\alpha,\alpha',\alpha'')\in (\mathcal{H}^T)^{3d}$ \[\langle v(\alpha),\alpha'-\alpha''\rangle^T= l(\alpha)(\alpha',\alpha'').\]
\end{enumerate}
\end{lemma}
\begin{proof}
Since the system \eqref{parametrized fbsde} is a decoupled forward backward system with Lipschitz coefficients the wellposedness follows once again from \cite{zhang2017backward} Theorem 4.3.1. Then, the first statement follows from an application of Gronwall's lemma. The second statement is just a consequence of the definition of $(U^\alpha_t)_{t\in [0,T]}$ for a given $\alpha \in \left(\mathcal{H}^T\right)^d$ and Ito's lemma. 
\end{proof}
We may now state the corresponding theorem to Theorem \ref{thm: convergence mfgc} in the case of FBSDEs 
\begin{thm}
    \label{thm: convergence fbsde}
Let v be defined as in Lemma \ref{lemma: def v fbsde}, for a given initial condition $\alpha^1\in \left(\mathcal{H}^T\right)^{d}$ and $n\geq 1$ we introduce the sequences
\[\left\{\begin{array}{c}
     \alpha^{n+\frac{1}{2}}=\alpha^n-\gamma v(\alpha^n),  \\
     \alpha^{n+1}=\alpha^n-\gamma v(\alpha^{n+\frac{1}{2}}),
\end{array}\right. \]
If 
\[\gamma\leq \frac{1}{\|v\|_{Lip}},\]
then under Hypotheses \ref{hyp: L2 monotonicity}, and \ref{hyp: lipschitz fbsde}, letting 
\[\forall t\in [0,T],\quad \bar{U}^n_t=\frac{1}{n}\sum_{i=1}^n \tilde{F}^{-1}_u(X^{\alpha^{i+\frac{1}{2}}}_t,\alpha^{i+\frac{1}{2}}_t),\]
the following holds 
\begin{enumerate}
    \item[-] \[\forall n\geq 1,\quad \int_0^T \|U_t-\bar{U}^n_t\|^2dt\leq \frac{1}{2\gamma c_L n}\int_0^T\|\tilde{F}(X_t,U_t)-\alpha^1_t\|^2dt.\]
    \item[-] Letting  \[\forall t\in [0,T],\quad \bar{X}^n_t=X_0-\int_0^t \tilde{F}(\bar{X}^n_s,\bar{U}^n_s)ds+\sqrt{2\sigma} B_t,\] there exists a constant $C$ depending only on $T$ and $\|\tilde{F}\|_{Lip}$ such that 
    \[\int_0^T\left\|\bar{X}^n_t-X_t\right\|^2dt\leq \frac{C}{2\gamma c_L n}\int_0^T\left\|\tilde{F}(X_t,U_t)-\alpha^1_t\right\|^2dt.\] 
\end{enumerate}
\end{thm}
\begin{proof}
    We follow the same idea as in Theorem \ref{thm: convergence mfgc} of using the monotone functional $v$ combined with Lemma \ref{GEG inequality}. Fixing $\alpha\in \left(\mathcal{H}^T\right)^d$ this yields
    \begin{gather*} n\langle v(\alpha),\bar{\alpha}_n-\alpha\rangle^T+c_F\sum_{i=1}^n \int_0^T\left\|\tilde{F}^{-1}_u(X^\alpha_t,\alpha_t)-\tilde{F}^{-1}_u(X^{\alpha^{i+\frac{1}{2}}}_t,\alpha^{i+\frac{1}{2}}_t)\right\|^2dt\\ \leq\sum_{i=1}^n\langle v(\alpha^{i+\frac{1}{2}}),\alpha^{i+\frac{1}{2}}-\alpha\rangle^T\leq \frac{\|\alpha-\alpha^1\|^2_T}{2\gamma},\end{gather*}
    for 
    \[\bar{\alpha}_n=\frac{1}{n}\sum_{i=1}^n \alpha^{i+\frac{1}{2}}.\]
    Evaluating this expression for $\alpha^*=\left(\tilde{F}(X_t,U_t)\right)_{t\in [0,T]}$ where $(X_t,U_t,Z_t)_{t\in [0,T]}$ is the unique strong solution to \eqref{FBSDE} we get 
    \begin{align*}n\int_0^T \left\|U_t-\frac{1}{n}\sum_{i=1}^n\tilde{F}^{-1}_u(X^{\alpha^{i+\frac{1}{2}}}_t,\alpha^{i+\frac{1}{2}}_t)\right\|^2dt&\leq \sum_{i=1}^n \int_0^T \left\|U_t-\tilde{F}^{-1}_u(X^{\alpha^{i+\frac{1}{2}}}_t,\alpha^{i+\frac{1}{2}}_t)\right\|^2dt\\ &\leq \frac{1}{2\gamma c_F}\int_0^T\left\|\tilde{F}(X_t,U_t)-\alpha_t\right\|^2dt\end{align*}
    proving the first claim. The second claim follows naturally from an application of Gronwall Lemma. 
\end{proof}
\begin{remarque}
    \label{remarque: trouble if inverse not known}
    Clearly this procedure requires precise knowledge of $\tilde{F}^{-1}_u$ to lead to a reasonable algorithm. Although this was not a problem in mean field games of control, as no inverse was used in the definition of $v$, it appears that this algorithm is not very suitable to compute the solution of Forward backward systems where the forward driver $F$ depends on the law of the backward process $U$. In this case the Hilbertian inverse is in general not trivial, whereas when $F$ doesn't depend on the law of the backward process, $\tilde{F}^{-1}_u$ is just 
    \[(X,U)\mapsto F^{-1}_u(X,U,\mathcal{L}(X))\]
    where $u\mapsto F^{-1}_u(x,u,\mu)$ indicates the inverse of $u\mapsto F(x,u,\mu)$ in $\reels^d$ for a given pair $(x,\mu)\in \reels^d\times \mathcal{P}_2(\reels^d)$.
\end{remarque}
\begin{remarque}
    Let us observe that, if the following monotonicity condition 
    \begin{gather*}
\exists c_F>0, \quad \forall (X,Y,U,V)\in \mathcal{H}^{4d}, \quad \\
\langle  g(X,\mathcal{L}(X))-g(Y,\mathcal{L}(Y)),X-Y\rangle \geq 0,\\
\langle \left(\begin{array}{c}
F\\
G
\end{array}\right)
(X,Y,\mathcal{L}(X,Y))-\left(\begin{array}{c}
F\\
G
\end{array}\right)(Y,V,\mathcal{L}(Y,V)),\left(\begin{array}{c}
     X-Y  \\
     U-V
\end{array}\right)\rangle\\ \geq c_F\|F(X,U,\mathcal{L}(X,U))-F(Y,V,\mathcal{L}(Y,V))\|^2,
\end{gather*}
is satisfied instead. Then 
\[\forall (\alpha,\alpha')\in \left(\mathcal{H}^T\right)^{2d}, \quad \langle v(\alpha)-v(\alpha'),\alpha-\alpha'\rangle^T\geq c_F\|\alpha-\alpha'\|^2_T,\]
and an analogue to Theorem \ref{lemma: exponential convergence}, on the exponential convergence to the solution of the FBSDE \eqref{FBSDE} follows naturally. the Wellposedness of \eqref{FBSDE} under such monotonicity condition is proven in the appendix \eqref{lemma: wellposedness fbsde strong monotonicity on F}.
In particular, this monotonicity condition is always satisfied whenever coefficients are Lipschitz and the pair $(F,G)$ is strongly monotone in $(X,U)$ 
\end{remarque}
In the case of mean field games, there exists a canonical monotone functional coming from the optimal control representation of the problem. In a sense this corresponds to a potential regime compared to general FBSDEs. In this latter case, it is unclear to us whether the parametrization we introduce \eqref{parametrized fbsde} yields the best results. Although it certainly feels like a natural extension to MFGs, depending on the problem there might exist parametrizations leading to much better convergence rates.

Before ending this section we present a lemma that will prove useful to estimate the convergence rate of a sequence of approximate solutions 
\begin{lemma}
    \label{lemma: error in function of v(alpha)}
   Under Hypotheses \ref{hyp: lipschitz fbsde} and \ref{hyp: L2 monotonicity}, letting $(X_s,U_s,Z_s)_{s\in [0,T]}$ be the unique strong solution to \eqref{FBSDE}, there exists a constant $C$ depending only on $c_F,T$ and the Lipschitz constants of $F,G,g$ such that for any control $\alpha\in \mathcal{H}^T$ 
   \[\forall t\in [0,T], \quad \| U^\alpha_t-U_T\|, \|X^\alpha_t-X_t\|\leq C \|v(\alpha)\|_T.\]
\end{lemma}
\begin{proof}
    By definition of $v(\alpha)$ 
    \[\forall t\in [0,T],\quad \alpha_t=F(X^\alpha_t,U^\alpha_t+v(\alpha)_t,\mathcal{L}(X^\alpha_t,U^\alpha_t+v(\alpha)_t)).\]
    Naturally, this is equivalent to the fact that the pair $(X^\alpha_t,U^\alpha_t)_{t\in [0,T]}$ solves
    \begin{equation}
\label{fbsde error v alpha}
\left\{\begin{array}{l}
\displaystyle X^\alpha_t=X_0-\int_0^tF(X^\alpha_s,U^\alpha_s+v(\alpha)_s,\mathcal{L}(X^\alpha_s,U^\alpha_s+v(\alpha)_s)) ds +\sqrt{2\sigma} B_t,\\
\displaystyle U^\alpha_t=U^\alpha_T+\int_t^T G(X^\alpha_s,U^\alpha_s+v(\alpha)_s,\mathcal{L}(X^\alpha_s,U^\alpha_s+v(\alpha)_s))ds-\int_t^T Z^\alpha_s dB_s,\\
U^\alpha_T=g(X^\alpha_T,\mathcal{L}(X^\alpha_T)).
\end{array}\right.
\end{equation}
Let $(X_t,U_t)_{t\in [0,T]}$ be the unique solution to \eqref{FBSDE} we define 
\[\forall t\in [0,T], \quad V_t=(U_t-U^\alpha_t)\cdot(X_t-X^\alpha_t).\]
By Ito's Lemma, 
\[\forall t\in [0,T], \quad V_T=V_t-\int_t^T \left(\Delta F_s\cdot (U_s-U^\alpha_s)+\Delta G_s\cdot(X_s-X^\alpha_s)\right)ds,\]
with the notation
\[ \forall s\in [0,T], \quad \Delta F_s=F(X_s,U_s,\mathcal{L}(X_s,U_s))-F(X^\alpha_s,U^\alpha_s+v(\alpha)_s,\mathcal{L}(X^\alpha_s,U^\alpha_s+v(\alpha)_s)),\]
and an analogue definition for $(\Delta G_s)_{s\in [0,T]}$. By the monotonicity of coefficients and their Lipschitz regularity, there exists a constant $C_{Lip}$ depending only on the Lipschitz norms of $F,G$ such that 
\begin{equation}
    \label{eq: using the monotonicity of initial condition estimate on v(alpha)}
    \forall t\in [0,T],\quad 0\leq \esp{V_T}\leq \esp{V_t}+\underbrace{C_{Lip}\int_0^T \|v(\alpha)_s\|\left(\|X_s-X^\alpha_s\|+\|U_s-U^\alpha_s\|\right)ds}_{C^\alpha_T}.\end{equation}
Using the strong monotonocity of $(F,G)$ we deduce that for any $t\in [0,T]$
\[c_F\int_0^t\|U_s-U^\alpha_s\|^2ds+\esp{V_t}\leq \underbrace{V_0}_{=0}+C_{Lip}\int_0^t \|v(\alpha)_s\|\left(\|X_s-X^\alpha_s\|+\|U_s-U^\alpha_s\|\right)ds.\]
Using \eqref{eq: using the monotonicity of initial condition estimate on v(alpha)}, it follows that 
\begin{equation}
    \label{eq: u in function of x estimate v(alpha)}
    \forall t\in [0,T],\quad \frac{c_F}{2}\int_0^t\|U_s-U^\alpha_s\|^2ds\leq C_{Lip}^2\left(\frac{2}{c_F}+1\right)\|v(\alpha)\|^2 _T+\int_0^t \|X_s-X^\alpha_s\|^2ds+C^\alpha_T.
\end{equation}
Since 
\[\forall t\in [0,t], \quad \|X^\alpha_t-X_t\|^2\leq C_{Lip}\int_0^t \left(\|X^\alpha_s-X_s\|^2+2\|U^\alpha_s-U_s\|^2+2\|v(\alpha)_s\|^2\right)ds,\]
estimating the term depending on $(U^\alpha_s-U_s)_{s\in [0,t]}$ with \eqref{eq: u in function of x estimate v(alpha)}, we get by an application of Gronwall's lemma that there exists a constant $C$ depending on $c_F,T$ and the Lipschitz constants of $(F,G)$ such that 
\begin{equation}
    \label{gronwall x in u v(alpha)}
    \forall t\in [0,T], \quad \|X^\alpha_t-X_t\|^2\leq C\left(\|v(\alpha)\|^2_T+C^\alpha_T\right).\end{equation}
Integrating from $0$ to $T$ and using Cauchy-schwartz inequality we get 
\[\int_0^T \|X_s-X^\alpha_s\|^2ds\leq \left(\|v(\alpha)\|^2_T(C+C_{Lip}^2C^2)+C\int_0^T\|v(\alpha)_s\|\| U^\alpha_s-U_s\|ds\right).\]
In particular, by plugging back this estimate in \eqref{gronwall x in u v(alpha)}, 
\[\forall t\in [0,T], \quad \|X^\alpha_t-X_t\|^2\leq C'\left(\|v(\alpha)\|^2_T+\int_0^T\|v(\alpha)_s\|\| U^\alpha_s-U_s\|ds\right),\]
for a constant $C'$ with the same dependencies as $C$. Using this estimate, we can estimate $\int_0^T \|U^\alpha_s-U_s\|^2 ds$ in a similar fashion and we conclude to the existence of a constant $C{''}$ depending on $c_F,T$,$\|(F,G,g)\|_{Lip}$ only such that 
\[\forall t\in [0,T], \quad \|X_t-X^\alpha_t\|, \|U_t-U^\alpha_t\|\leq C^{''} \|v(\alpha)\|_T. \]
\end{proof}
\begin{remarque}
    \label{remarque: strong stability of fbsde}
    Let us remark at this point that a slightly more general version of this statement is possible. Indeed this Lemma is exactly a consequence of the strong stability of strongly monotone FBSDE in $L^2$. Considering the system 
    \begin{equation}
        \label{eq: perturbized system}
\left\{
\begin{array}{l}
     \displaystyle X^\eta_t=X_0-\int_0^t  \left(F(X^\eta_s,U^\eta_s,\mathcal{L}(X^\eta_s,U^\eta_s))+\eta^x_s\right)ds+\sqrt{2\sigma}B_t, \\
    \displaystyle  U^\eta_t= g(X^\eta_T,\mathcal{L}(X^\eta_T))+\gamma_T+\int_t^T \left(G(X^\eta_s,U^\eta_s,\mathcal{L}\left(X^\eta_s,U^\eta_s\right))+\eta^u_s\right)ds-\int_t^T Z_sdB_s,
\end{array}
\right.
\end{equation}
for $(\eta^x_t,\eta^u_t)_{t\in [0,T]}\in \left(\mathcal{H}^T\right)^{2d}$ and a $\mathcal{F}_T-$measurable $\gamma_T\in \mathcal{H}^d$. There exists a constant $C_{stab}$ depending on $c_F,T,\|(F,G,g)\|_{Lip}$ only such that if there exist a strong solution to \eqref{eq: perturbized system}, then 
\[\esp{\sup_{[0,T]}|X_t-X^\eta_t|^2+\sup_{[0,T]}|U^\eta_t-U_t|^2}\leq C_{stab}\left(\|\gamma_T\|^2+\int_0^T \|(\eta^x_t,\eta^u_t)\|^2dt\right).\]
This result follow from the exact same argument. Formally, the strong monotonicity of the system allows to decouple the dependencies, yielding stability estimates which are standard for decoupled FBSDE but with worse constants, depending in particular on $\frac{1}{c_F}$. 
\end{remarque}

\begin{subsection}{Decreasing step size}
So far, we have been working with a constant step size $\gamma$. This yields a convergence of the forward backward system in $\left(\mathcal{H}^T\right)^d$ with a rate of $O\left(\frac{1}{\sqrt{n}}\right)$. Altough choosing a decreasing step size leads to a worse converge rate, it allows to weaken some assumption. In particular, the step size can then be chosen independently of the Lipschitz norm of $v$. This is particularly interesting for FBSDEs posed on large time intervals since we expect that $\|v\|_{Lip}$ grows linearly with the horizon $T$. 
\begin{corol}
    Let v be defined as in \ref{lemma: def v fbsde}, for a given initial condition $\alpha^1\in \left(\mathcal{H}^T\right)^{d}$ and $n\geq 1$ we introduce the sequences
\[\left\{\begin{array}{c}
     \alpha^{n+\frac{1}{2}}=\alpha^n-\gamma_n v(\alpha^n),  \\
     Y^{n+1}=Y^n-v(\alpha^{n+\frac{1}{2}}),\\
     \alpha^{n+1}=\gamma_{n+1} Y^{n+1},
\end{array}\right., \]
with 
\[Y^1=\frac{1}{\gamma_1}\alpha^1.\]
If $(\gamma_n)_{n\in \mathbb{N}}\subset \reels^+$ is a sequence of decreasing step such that there exists $C\in \reels^+, \delta \in (0,1)$ such that 
\[\forall n\in \mathbb{N}, \quad \frac{1}{\gamma_n}\leq C_\gamma n^\delta,\]
and 
\[\exists n_0, \quad \gamma_{n_0}\leq \|v\|_{Lip},\]
then under Hypotheses \ref{hyp: L2 monotonicity}, \ref{hyp: lipschitz fbsde}, letting 
\[\bar{U}^n=\frac{1}{n}\sum_{i=1}^n \tilde{F}^{-1}_u(X^{\alpha_{i+\frac{1}{2}}},\alpha_{i+\frac{1}{2}}),\]
and $(X_t,U_t,Z_t)_{t\in [0,T]}$ be the unique strong solution to \eqref{FBSDE}, the following holds
\begin{enumerate}
    \item[-] There exists a constant $C_v$ depending on the coefficients and $(\gamma_n)_{n\in \mathbb{N}}$ such that for any $n\geq 1$ \[\int_0^T\left\|U_t -\bar{U}^n_t\right\|^2dt\leq  \frac{1}{2\gamma_1c_f n }\int_0^T\left\|\tilde{F}(X_t,U_t)-\alpha^1_t\right\|^2dt+\frac{C_\gamma}{2n^{1-\delta}}\int_0^T\|U_t\|^2dt+\frac{C_v}{c_Fn}.\]
    \item[-] Letting  \[\forall t\in [0,T],\quad \bar{X}^n_t=X_0-\int_0^t \tilde{F}(\bar{X}^n_s,\bar{U}^n_s)ds+\sqrt{2\sigma} B_t,\] there exists a constant $C$ depending only on $T$ and $\|\tilde{F}\|_{Lip}$ such that 
    \[\int_0^T\| \bar{X}^n_t-X_t\|^2dt\leq C\int_0^T\left\|U_t -\bar{U}^n_t\right\|^2dt.\] 
\end{enumerate}
\end{corol}
\begin{proof}
In spirit the proof is quite similar to that of Theorem \ref{thm: convergence fbsde}: by applying Lemma \ref{GEG inequality} in $\left(\mathcal{H}^T\right)^d$ with 
\[x=\alpha^*=\left(\tilde{F}(X_t,U_t)\right)_{t\in [0,T]},\]
we find that 
     \begin{gather*}c_Fn\int_0^T\left\| U_t-\frac{1}{n}\sum_{i=1}^n\tilde{F}^{-1}_u(X^{\alpha_{i+\frac{1}{2}}}_t,\alpha^{i+\frac{1}{2}}_t)\right\|^2dt \\\leq \frac{1}{2\gamma_1 }\int_0^T \|\tilde{F}(X_t,U_t)-\alpha^1_t\|^2dt+\frac{1}{2\gamma_{n+1}}\int_0^T\|U_t\|^2dt\\
        \quad +\underbrace{\frac{1}{2}\sum_{i=1}^n \left(\gamma_i \|v(\alpha^{i+\frac{1}{2}})-v(\alpha^i)\|^2_T-\frac{1}{\gamma_i}\|\alpha^{i+\frac{1}{2}}-\alpha^i\|^2_T\right)}_{E_n}.
    \end{gather*}
Letting
\[C_v=\frac{1}{2}\sum_{i=1}^{n_0}\left(\gamma_i \|v(\alpha_{i+\frac{1}{2}})-v(\alpha_i)\|^2_T-\frac{1}{\gamma_i}\|\alpha_{i+\frac{1}{2}}-\alpha_i\|^2_T\right),\]
from our assumption on the sequence $(\gamma_n)_{n\in \mathbb{N}}$ 
\[\forall n\in \mathbb{N}, \quad E_n\leq C_v.\]
It follows that 
\begin{gather*}\int_0^T\left\| U_t-\frac{1}{n}\sum_{i=1}^n\tilde{F}^{-1}_u(X^{\alpha^{i+\frac{1}{2}}}_t,\alpha^{i+\frac{1}{2}}_t)\right\|^2dt\\
    \leq \frac{1}{2\gamma_1c_f n }\int_0^T\|\tilde{F}(X_t,U_t)-\alpha^1_t\|^2dt+\frac{C_\gamma}{2n^{1-\delta}}\int_0^T\|U_t\|^2dt+\frac{C_v}{c_F n}.\end{gather*}
The rest of the proof is a straightforward consequence of this result
\end{proof}
\begin{remarque}
    the condition 
\[\exists n_0, \quad \gamma_{n_0}\leq \frac{1}{\|v\|_{Lip}},\]
is always true as soon as the sequence tends to 0. If the operator $v$ is only assumed to be $\eta$ Holder for some $\eta\in (0,1)$, then this result can also be adapted by requiring that
\[\sum_{n=1}^{+\infty} \gamma_n^\frac{1+\eta}{1-\eta}<+\infty,\]
instead. By an application of Young's inequality
\[\sum_{i=1}^n \gamma_i\|\alpha_{i+\frac{1}{2}}-\alpha_i\|^{2\eta}\leq \frac{1}{1-\eta}\sum_{i=1}^n \gamma_i^\frac{1+\eta}{1-\eta}+\frac{1}{\eta} \sum_{i=1}^n \frac{1}{\gamma_i}\|\alpha_{i+\frac{1}{2}}-\alpha_i\|^{2}.\]
For monotone FBSDEs with Holder coefficients we refer to \cite{monotone-sol-meynard}
\end{remarque}
\subsection{Outside of the mean field regime}
Let us mention that this method is still valid outside of the mean field regime, namely this allows to compute the characteristics of system of PDEs of the form 
\begin{equation}
    \label{eq: non linear transport finite dimension}
    \partial_t U+F(x,U)\cdot \nabla_x U=G(x,U),
\end{equation}
in the monotone regime. For mean field games, this corresponds to the study of finite state space MFGs \cite{Lions-college}. Computing solutions to this system can be challenging as the regularity of this PDE stems from its monotonicity. In general this means that either the interval of time considered must be sufficiently small, or the approximation used must be monotonicity preserving. Finally even in the so called potential regime, that is, when \eqref{eq: non linear transport finite dimension} arises from the gradient of a convex HJB equation \cite{convex-HJB}
\[\partial_t u+H(x,\nabla_x u)=0,\]
the method we propose appears new. Although it can only be applied to the framework of convex HJB equation compared to more general method such as shooting algorithms \cite{numerical-HJB}, it exhibits global convergence independently of the chosen initial control. 

Another point which is relevant in the finite dimensional setting is that $v$ can be monotone without requiring that all coefficients are. Consider the following HJB equation for simplicity
\begin{equation}
\label{hjb}
\left\{
    \begin{array}{c}
       \partial_t u+H(x,\nabla_x u)-\sigma \Delta_x u=0, \quad (t,x)\in (0,T)\times \reels^d,\\ 
       u(T,x)=g(x) \quad \forall x\in \reels^d.
    \end{array}
\right.
\end{equation}
We place ourselve under the following assumptions
\begin{hyp}
    \label{hyp: hjb equation}
Let $L$ be the fenchel conjugate of $H$ in \eqref{hjb}, 
\begin{enumerate}
    \item[-]there exists two constant $c_\alpha>0, c_x\geq 0$ such that \[\forall (x,\alpha)\in \reels^{2d} \quad \left(\begin{array}{c} D^2_x L \quad D^2_{x,\alpha} L\\D^2_{\alpha,x} \quad D^2_\alpha L
\end{array}. \right)(x,\alpha)\geq  \left(\begin{array}{c} -c_x I_d\quad  0\\ 0  \quad c_\alpha I_d\end{array}. \right),\]in the sense of distributions.
    \item[-] there exists $c_g\geq 0$ such that \[\forall (x,y)\in \reels^{2d}, \quad (g(x)-g(y))\cdot (x-y)\geq -c_g |x-y|^2.\]
    \item[-] $\nabla_xL,\nabla_\alpha L,\nabla_x g$ are lipschitz.
    \item[-] $\sigma>0$.  
\end{enumerate}
\end{hyp}
Those assumptions, especially on $\nabla_\alpha L$ are much stronger than what is needed for the wellposedness (in a classical sense) of \eqref{hjb}, nevertheless they are quite standard in optimal control and it is a convenient setting to state the following: 
\begin{lemma}
    Under Hypothesis \ref{hyp: hjb equation}, there exists a unique solution $u\in C([0,T]\times \reels^d)\cup C^{1,2}((0,T)\times \reels^d)$ to \eqref{hjb}. Moreover $x\mapsto \nabla_x u(t,x)$ is Lipschitz uniformly in $t\in [0,T]$. In particular for any initial condition $X\in \mathcal{H}$ there exists a unique solution to the forward backward system 
    \begin{equation}
        \label{fbsde hjb equation}
        \left\{ \begin{array}{c}
        dX_t=-D_p H(X_t,U_t)dt+\sqrt{2\sigma} dB_t ,\quad X_0=X,\\
        dU_t=D_x H(X_t,U_t)dt-Z_tdB_t, \quad U_T=\nabla_x g(X_T),
    \end{array}\right.,
\end{equation}
whose decoupling field is given by $(t,x)\mapsto \nabla_x u(t,x)$.
\end{lemma}
\begin{proof}
In this setting the uniqueness of a solution to \eqref{hjb} follows from the theory of viscosity solutions \cite{crandall1992users}. The existence is also quite standard. For example, it suffices to use any short time existence result \cite{lipschitz-sol} combined with schauder type estimates \cite{krylov1996holder}. The fact that $\nabla_x u$ is the decoupling field associated to the FBSDE \eqref{fbsde hjb equation} is then simply a consequence of Ito's lemma. 
\end{proof}
Since the existence of a solution to \eqref{fbsde hjb equation} for any initial condition has been established we now show the convergence of the algorithm introduced in this section 
\begin{lemma}
    Fix an initial condition $X_0\in \mathcal{H}^d$ and let $v$ be defined as in \ref{lemma: def v mfgc} for $L,g$. Under Hypothesis \ref{hyp: hjb equation}, if 
    \[c_v=c_\alpha-T(c_x+c_g)>0,\]
    and 
    \[\gamma \leq \min \left(\frac{1}{2\|v\|_{Lip}},\frac{c_v}{\|v\|_{Lip}^2}\right),\]

    then there exists a $\lambda \in (0,1)$ depending only on $c_v,\gamma, \|v\|_{Lip}$ such that for any $\alpha_1\in (\mathcal{H}^T)^d$ letting 
\[\left\{\begin{array}{c}
     \alpha_{n+\frac{1}{2}}=\alpha_n-\gamma v(\alpha_n),  \\
     \alpha_{n+1}=\alpha_n-\gamma v(\alpha_{n+\frac{1}{2}}).
\end{array}\right. \]
the following holds 
\[\forall n\in \mathbb{N}, \quad n\geq 1 ,\quad \|\alpha_n-\alpha^*\|_T\leq \lambda^n\|\alpha_1-\alpha^*\|_T\]
for $\alpha^*=\left(D_p H(X_t,U_t)\right)_{t\in [0,T]}$ the optimal control associated to \eqref{fbsde hjb equation}.
\end{lemma}
\begin{proof}
    Since coefficients are all Lipschitz, it is evident that $v$ is Lipschitz. Now let us take two controls $\alpha,\alpha'$, using Hypothesis \ref{hyp: hjb equation}, we get
    \[\langle v(\alpha)-v(\alpha'),\alpha-\alpha'\rangle^T\geq c_\alpha \|\alpha-\alpha'\|^2_T-c_x\int_0^T \|X^\alpha_t-X^{\alpha'}_t\|^2dt-c_g \|X^\alpha_T-X^{\alpha'_T}\|^2.  \]
    By definition 
    \[\forall t\in [0,T], \quad \|X^\alpha_t-X^{\alpha'}_t\|^2\leq \|\alpha-\alpha'\|^2_T.\]
    As such 
    \[\langle v(\alpha)-v(\alpha'),\alpha-\alpha'\rangle^T\geq c_v\|\alpha-\alpha'\|^2_T.\]
    Since this is true for any controls $\alpha,\alpha'$, $v$ is Lipschitz continuous and strongly monotone, and we are in the conditions of applications of Theorem \eqref{lemma: exponential convergence}.
\end{proof}
Let us remark that even in mean field games, $v$ may be monotone even when coefficients aren't. Indeed the above result is adapted easily to the mean field setting on short intervals of time. So long as the monotonicity of $v$ is conserved, the existence and uniqueness of a solution are not particularly hard to prove, this phenomenon is sometimes refered to as the semi-monotone regime in short time \cite{propagation-of-monotonicity}. 
 \end{subsection}
\section{FBSDEs with common noise}
\label{section: common noise}
We now turn to the problem of FBSDEs with common noise 
\begin{equation}
    \label{eq: fbsde with common noise}
\left\{
\begin{array}{l}
     \displaystyle X_t=X_0-\int_0^t  F(X_s,p_s,U_s,\mathcal{L}(X_s,U_s|\mathcal{F}^0_s))ds+\sqrt{2\sigma}B_t, \\
    \displaystyle  U_t= U_T+\int_t^T G(X_s,p_s,U_s,\mathcal{L}\left(X_sU_s|\mathcal{F}^0_s\right))ds-\int_t^T Z_sd(B_s,W_s),\\
    \displaystyle p_t=p_0-\int_0^t b(p_s)ds+\sqrt{2\sigma^0}W_t,\\
    U_T=g(X_T,p_T,\mathcal{L}(X_T|\mathcal{F}^0_T)),
\end{array}
\right.
\end{equation}
for $X_0\in \mathcal{H}^d,q_0\in \mathcal{H}^{d^0}$ initial conditions independent of each others and of $(W_t,B_t)_{t\geq 0}$. In this case the filtration associated to common noise is $(\mathcal{F}^0_t)_{t\in [0,T]}$ which is the augmented filtration associated to $(\mathcal{\tilde{F}}^0_t)_{t\in [0,T]}$ defined by 
\[\mathcal{\tilde{F}}^0_t=\sigma(q_0,(W_s)_{s\leq t}).\]
\begin{remarque}
We insist on the fact that this class of FBSDEs includes the case of additive common noise often studied in mean field games. Indeed, consider the following problem
\begin{equation}
    \label{eq: fbsde with additive common noise}
\left\{
\begin{array}{l}
     \displaystyle X_t=X_0-\int_0^t  F(X_s,U_s,\mathcal{L}(X_s,U_s|\mathcal{F}^0_s))ds+\sqrt{2\sigma}B_t+\sqrt{2\sigma^0}W_t, \\
    \displaystyle  U_t= U_T+\int_t^T G(X_s,U_s,\mathcal{L}\left(X_s,U_s|\mathcal{F}^0_s\right))ds-\int_t^T Z_sd(B_s,W_s),\\
    U_T=g(X_T,\mathcal{L}(X_T|\mathcal{F}^0_T)),
\end{array}
\right.
\end{equation}
Introducing the two processes 
\[\forall t\in [0,T], \quad 
\left\{
    \begin{array}{c}
       p_t=\sqrt{2\sigma^0}W_t,\\
       Y_t=X_t-p_t, 
    \end{array}
\right.
\]
Formally, the triple $(Y_s,p_s,Z_s)_{s\in [0,T]}$ is a solution to 
\begin{equation*}
\left\{
\begin{array}{l}
     \displaystyle Y_t=X_0-\int_0^t  F(Y_s+p_s,U_s,T(\cdot,p_s)_{\#}\mathcal{L}(Y_s,U_s|\mathcal{F}^0_s))ds+\sqrt{2\sigma}B_t, \\
     \displaystyle p_t=\sqrt{2\sigma^0}W_t,\\
    \displaystyle  U_t= U_T+\int_t^T G(Y_s+p_s,U_s, T(\cdot,p_s)_{\#}\mathcal{L}(Y_s,U_s|\mathcal{F}^0_s))ds-\int_t^T Z_sd(B_s,W_s),\\
    U_T=g(Y_T+p_T,(id_{\reels^d}+p_T)_{\#}\mathcal{L}(X_T|\mathcal{F}^0_T)),
\end{array}
\right.
\end{equation*}
where $f_{\#}\mu$ indicates the pushforward of a measure $\mu$ by the function $f$ and 
\[(id_{\reels^d}+p): x\mapsto x+p, \quad T(\cdot,p):(x,y)\mapsto (x+p,y).\]
In particular if $(F,G,g)$ are Lipschitz functions on $\reels^{2d}\times \mathcal{P}_2(\reels^{2d})$ then the coefficients are still Lipschitz in all arguments after this transformation. A rigorous equivalence between these two formulations is established in \cite{common-noise-in-MFG} Lemma 5.3.
\end{remarque}
 We work under the following assumption 
\begin{hyp}
    \label{hyp: FBSDE with common noise}
    The coefficients $(g,F,G,b)$ are such that 
\begin{enumerate}
    \item[-] Hypothesis \ref{hyp: L2 monotonicity} is satisfied by 
    \[(x,u,m)\mapsto (g(x,p,\pi_dm),F(x,p,u,m),G(x,p,u,m))\]
    uniformly in $p\in \reels^{d^0}$.
    \item[-] $(F,G,g,b)$ are Lipschitz in all their arguments for $\mathcal{W}_2$ in the measure argument. 
\end{enumerate}
\end{hyp}
It has already been observed in \cite{noise-add-variable}, that the addition of such common noise does not perturb the monotonicity of the system. Hence results of wellposedness for monotone FBSDEs are easily adapted to this particular system where the common noise $(p_t)_{t\geq 0}$ evolves independently of the rest of the dynamics
\begin{lemma}
    \label{lemma: existence with common noise}
Under Hypothesis \ref{hyp: FBSDE with common noise}, for any initial condition $X_0\in \mathcal{H}^d,q_0\in \mathcal{H}^{d^0}$, independent of each other and of $(W_t,B_t)_{t\geq 0}$, there exists a unique strong solution to the FBSDE \eqref{eq: fbsde with common noise} 
\end{lemma}
\begin{proof}
For a proof of this result based on the existence of a Lipschitz decoupling field to the FBSDE \eqref{eq: fbsde with common noise}, we refer to \cite{monotone-sol-meynard} Theorem 3.33 and the subsection 3.4.1 of said work. 
\end{proof}
Let us remark that under the standing assumptions, the function 
\[U\mapsto F(X,p,U,\mathcal{L}(X,U)),\]
is invertible on $\mathcal{H}^d$ for any $(X,p)\in \mathcal{H}^{d}\times \reels^{d^0}$ by the monotonicity of $F$, with this Hilbertian inverse being denoted by $\tilde{F}^{-1}_u(X,p,\cdot)$. In particular following \cite{monotone-sol-meynard} Lemma 3.10 there exists a Lipschitz function $F_{inv}:\reels^d\times\reels^{d^0}\times \reels^d\times \mathcal{P}_2(\reels^{2d})\to \reels^d$ such that 
\[\forall (X,\alpha,p)\in \mathcal{H}^{2d}\times \reels^{d^0}, \quad F_{inv}(X,p,\alpha,\mathcal{L}(X,\alpha))=\tilde{F}^{-1}_u(X,p,\alpha)\]

Following the previous section, for a given $(\alpha_t)_{t\in [0,T]} \in \left(\mathcal{H}^T\right)^d$, we introduce the dynamics 
\begin{equation}
    \label{eq: linearized system with common}
 \left\{
\begin{array}{l}
    \displaystyle X^\alpha_t=X_0-\int_0^t \alpha_sds+\sqrt{2\sigma}B_t,\\
   \displaystyle  p_t=p_0-\int_0^t b(p_s)ds+\sqrt{2\sigma^0}W_t,\\
   \displaystyle  U^\alpha_t=U^\alpha_T-\int_t^T G_F(X^\alpha_s,p_s,\alpha_s,\mathcal{L}(X^\alpha_s,\alpha_s|\mathcal{F}^0_s))ds-Z_sd(B_s,W_s),\\
   U^\alpha_T=g(X^\alpha_T,p_T,\mathcal{L}(X^\alpha_T|\mathcal{F}^0_T)),
\end{array}
 \right.   
\end{equation}
where 
\begin{gather*}
\forall (x,p,a,X,\alpha)\in \reels^{2d+d^0}\times \mathcal{H}^{2d}, \quad \\
G_F(x,p,a,\mathcal{L}(X,\alpha))
=G(x,p,F_{inv}(x,p,a,\mathcal{L}(X,\alpha)),\mathcal{L}(X,F_{inv}(X,p,\alpha,\mathcal{L}(X,\alpha)))).
\end{gather*}
\begin{corol}
Under Hypothesis \ref{hyp: FBSDE with common noise}, for any $\alpha \in \left(\mathcal{H}^T\right)^d, X_0\in \mathcal{H}^d$ the system \eqref{eq: linearized system with common} has a unique solution $(X^\alpha_t,U^\alpha_t,Z^\alpha_t)_{t\in [0,T]}$. Moreover, letting
\begin{equation}
    \label{eq: v with common noise}
    v:\FuncDef{\left(\mathcal{H}^T\right)^d}{\left(\mathcal{H}^T\right)^d}{(\alpha_t)_{t\in [0,T]}}{ \left(F_{inv}(X^\alpha_s,p_s,\alpha_s,\mathcal{L}(X^\alpha_s,\alpha_s|\mathcal{F}^0_s))-U^\alpha_t\right)_{t\in [0,T]}},
\end{equation}
$v$ is Lipschitz on $\left(\mathcal{H}^T\right)^d$ and there exists a constant $C$ such that if 
\[\gamma\leq \frac{1}{\|v\|_{Lip}},\]
for any $\alpha^0\in \left(\mathcal{H}^T\right)^d$, the procedure 
\[\left\{\begin{array}{l}
     \alpha^{n+\frac{1}{2}}=\alpha^n-\gamma v(\alpha^n),  \\
     \alpha^{n+1}=\alpha^n-\gamma v(\alpha^{n+\frac{1}{2}}),\\
     \forall s\in [0,T], \quad \bar{U}^n_s=\frac{1}{n}\sum_{i=1}^n F_{inv}(X^{\alpha^{n+\frac{1}{2}}}_s,p_s,\alpha^{n+\frac{1}{2}}_s,\mathcal{L}(X^{\alpha^{n+\frac{1}{2}}}_s,\alpha^{n+\frac{1}{2}}_s|\mathcal{F}^0_s))
\end{array}\right. \]
satifies 
\[\forall n\geq 0, \quad C \int_0^T\| U_t-\bar{U}^n_t\|^2dt\leq\frac{ 1}{\gamma n}\int_0^T\|F_{inv}(X_t,p_t,U_t,\mathcal{L}(X_t,U_t|\mathcal{F}^0_t))-\alpha^0_t\|^2dt.\]
\end{corol}
\begin{proof}
This is almost a direct adaptation of Theorem \ref{thm: convergence fbsde}. The core of the proof remaining unchanged in the presence of an additional common noise. In particular it is straightforward to show that $v$ is still Lipschitz in this setting. Instead let us focus on showing that the monotonicity of $v$ is unchanged in this new problem. Let $\alpha,\alpha'\in \mathcal{H}^T$, by definition of the forward backward system \eqref{eq: fbsde with common noise}
\begin{gather*}\langle v(\alpha)-v(\alpha'),\alpha-\alpha'\rangle^T\\
    =\langle g(X^\alpha_T,p_T,\mathcal{L}(X^\alpha_T|\mathcal{F}^0_T))-g(X^{\alpha'}_T,p_T,\mathcal{L}(X^{\alpha'}_T|\mathcal{F}^0_T)),X^\alpha_T-X^(\alpha')_T\rangle\\
+\esp{\int_0^T \left(\Delta_s F_{inv}\right)\cdot (\alpha_s-\alpha'_s)ds+\int_0^T \left(\Delta_s G_F\right)\cdot (X^\alpha_s-X^{\alpha'}_s)ds}\\
\end{gather*}
with 
\begin{gather*}
\Delta_s F_{inv}=F_{inv}(X^\alpha_s,p_s,\alpha_s,\mathcal{L}(X^\alpha_s,\alpha_s|\mathcal{F}^0_s))- F_{inv}(X^{\alpha'}_s,p_s,\alpha'_s,\mathcal{L}(X^{\alpha'}_s,\alpha'_s|\mathcal{F}^0_s))\\
\Delta_s G_F=G_F(X^\alpha_s,p_s,\alpha_s,\mathcal{L}(X^\alpha_s,\alpha_s|\mathcal{F}^0_s))-G_F(X^{\alpha'}_s,p_s,\alpha'_s,\mathcal{L}(X^{\alpha'}_s,\alpha'_s|\mathcal{F}^0_s)).
\end{gather*}
For fixed $p\in \reels^{d^0}$, the monotonicity of the pair
\[(X,\alpha)\mapsto (F_{inv}(X,\alpha,p,\mathcal{L}(X,\alpha)),G_F(X,\alpha,p,\mathcal{L}(X,\alpha))),\]
has been established in Lemma \ref{lemma: property v fbsde}. Thus, using an equivalent definition of \\$L^2-$monotonicity viaprobability measures (instead of random variables), it is straightforward that 
\begin{gather*}\espcond{\int_0^T \left(\begin{array}{c}\Delta_s F_{inv}\\ \Delta_s G_F\end{array}\right)\cdot \left(\begin{array}{c}\alpha_s-\alpha'_s\\ X^\alpha_s-X^{\alpha'}_s\end{array}\right)ds}{\mathcal{F}^0_T}\\
    \geq c_F\espcond{\int_0^T\left|\Delta_s F_{inv}\right|^2 }{\mathcal{F}^0_T} \quad \text{a.s.}
\end{gather*}
The rest of the proof follows from arguments presented in the proof of Theorem \ref{thm: convergence fbsde} by using the above monotonicity estimate.
\end{proof}
\begin{remarque}
The adaptation of previously obtained convergence results to the case of FBSDEs with common noise is deceptively simple. Although the proof of convergence is quite similar, numerically this addition leads to significant changes in the algorithm considered.
\end{remarque}

\section{Numerical illustration}
\label{section: numerical}
\subsection{Presentation of a particle based method}
\subsubsection{Discretisation of the problem}
In this section we present some numerical results for the method introduced in this article. We first discretize the time interval considered $[0,T]$ with a time step $\Delta_t$ and work with discretized control $(\alpha_i)_{i\in \llbracket 1, N_t\rrbracket}=(\alpha(i \Delta_t))_{i\in \llbracket 1, N_t\rrbracket}$. The different processes are then computed with an explicit Euler scheme. Let us present the scheme more precisely in the case study of FBSDEs \eqref{FBSDE}
\begin{equation*}
\left\{
\begin{array}{l}
     \displaystyle X_t=X_0-\int_0^t  F(X_s,U_s,\mathcal{L}(X_s,U_s))ds+\sqrt{2\sigma}B_t, \\
    \displaystyle  U_t= g(X_T,\mathcal{L}(X_T))+\int_t^T G(X_s,U_s,\mathcal{L}\left(X_s,U_s\right))ds-\int_t^T Z_sdB_s.
\end{array}
\right.
\end{equation*}
We suppose that the map $F_{inv}:\reels^{2d}\times \mathcal{P}_2(\reels^{2d})\to \reels^d$ which is such that $\tilde{F}_{inv}=\tilde{F}^{-1}$ is known. The existence of such a map follows from Remark \ref{remarque: inverse of lift}, however as we already mentionned in Remark \ref{remarque: trouble if inverse not known}, computing it can be a challenge whenever $F$ depends on the law of $(U_t)_{t\in [0,T]}$. (although this is never a problem for mean field games, since in this case $F_{inv}=\nabla_\alpha L$ which is obviously known).

First we fix $N_p$ the number of paths simulated and introduce $(X^{i}_0)_{i\in \llbracket 1, N_p\rrbracket}\subset \reels^{N_p}$, where each $X^i_0$ is an independent realization of a random variable with law $\mathcal{L}(X_0)$. We also introduce $(G^{i}_j)_{(i,j)\in \llbracket 1, N_p\rrbracket\times \llbracket 1, N_t\rrbracket}$ a sequence of independent gaussian increments sampled from a normal distribution with mean $0$ and covariance $\sqrt{\Delta_t} I_d$. These will play the role of the Brownian increments 
\[W_{(i+1)\Delta_t}-W_{i\Delta_t}.\]
A discrete control $(\alpha^{i}_j)_{(i,j)\in \llbracket 1, N_p\rrbracket\times \llbracket 1, N_t\rrbracket}$ gives the control at the timestep $j$ along the path $i$.
We define $X^{i,0}=X^i_0$ and for $j\in  \llbracket 1, N_t\rrbracket$ 
\[X^{i}_j=X^{i}_{j-1}+\Delta_t \alpha^{i}_j+\sqrt{2\sigma}G^{i}_j.\]
To approximate the backward process $(U^\alpha_t)_{t\in [0,T]}$, we first define 
\[U^{i}_{N_t}=g(X^{i}_{N_t},\frac{1}{N_p}\sum_k\delta_{X^{k}_{N_t}}),\]
where $\delta_x$ indicates the dirac measure centered in $x\in \reels^d$
Then, letting 
\[\forall j\in \llbracket 1, N_t\rrbracket, \quad m_j=\frac{1}{N_p}\sum_k\delta_{(X^{k}_j,\alpha^{k}_j)},\]
indicate the empirical measure associated to the system and 
\[\forall j\in \llbracket 1, N_t\rrbracket, \quad m^F_j=\frac{1}{N_p}\sum_k\delta_{(X^{k}_j,F_{inv}(X^k_j\alpha^{k}_j,m_j))},\]
the sequence is constructed by backward induction for $j\in \llbracket 0, N_t-1\rrbracket$, 
\[U^{i}_j=\espcond{U^{i}_{j+1}+\Delta_t G(X^{i}_{j+1},F_{inv}(X^{i}_{j+1},\alpha^{i}_{j+1},m_{j+1}),m^F_{j+1})}{\mathcal{F}_j},\]
where $\mathcal{F}_j=\sigma\left((X_0^i,G^{i}_k,\alpha^{i}_k)_{i\in \llbracket 1, N_p\rrbracket,k\leq j}\right)$. The conditional expectation is approximated as follow, letting $(f_\omega)_{\omega \in \mathcal{X}}$ be a parametrized family of functions on some space $\mathcal{X}$, for fixed $j$ we approximate $(U^{i}_j)_{i\in \llbracket 1, N_p\rrbracket}$ with 
\[\forall i \in \{1,\cdots N_p\}, \quad U^{i}_j\approx f_{\omega^*}(X^{i}_j), \]
where 
\[\omega^*=\text{arginf}_{\omega \in \mathcal{X}} J(\omega),\]
for 
\begin{gather*}J(\omega)=\frac{1}{N_p}\sum_k \left|U^{k}_{j+1}+\Delta_t G(X^{k}_{j+1},F_{inv}(X^{k}_{j+1},\alpha^{k}_{j+1},m_{j+1}),m^F_{j+1})-f_\omega (X^{k}_j)\right|^2.\end{gather*}
Clearly this is based on the idea that for two random variables $X,Y\in \mathcal{H}$, 
\[\espcond{Y}{X}=\inf_{\text{measurable} f} \esp{|Y-f(X)|^2}.\]
The parametrized family $(f_\omega)_{\omega\in \mathcal{X}}$ can range from neural networks to any sufficiently rich regression method, for general results on regression based methods for BSDEs we refer to \cite{gobet-regression-BSDE}. Now that we have explained how the trajectories are calculated we introduce the following discretized operator 
\begin{gather*}v^{N_t,N_p}((X_0^i,G^{i}_j)_{(i,j)\in \llbracket 1, N_p\rrbracket\times \llbracket 1, N_t\rrbracket}):\\
    \FuncDef{\reels^{ N_p\times N_t\times d}}{\reels^{N_p\times N_t\times d}}{(\alpha^{i}_j)_{(i,j)\in \llbracket 1, N_p\rrbracket\times \llbracket 1, N_t\rrbracket}}{\left(F_{inv}(X^{i}_j,\alpha^{i}_j,m_j)-U^{i}_j\right)_{(i,j)\in \llbracket 1, N_p\rrbracket\times \llbracket 1, N_t\rrbracket}}.
\end{gather*}
We can now introduce the associated algorithm to solve the FBSDE \eqref{FBSDE}
\begin{algorithm}[H]
    \caption{Monotone FBSDE solver}\label{algo FBSDE}
    \begin{algorithmic}
        \label{algo 1}
        \STATE fix an error level $\varepsilon_0$. 
        \STATE fix a decreasing sequence of step $(\gamma_n)_{n\in \mathbb{N}}\subset \reels^+$. 
        \STATE $N_{iter}\leftarrow 0$
        \STATE Simulate $N_p$ independent realizations $(X_0^i)_{i\in \llbracket 1,N_p\rrbracket }$ of $X_0$. 
        \STATE Simulate $N_p\times N_t$ realizations $(G^{i,j})_{(i,j)\in \llbracket 1, N_p\rrbracket\times \llbracket 1, N_t\rrbracket}$ from independant gaussian random variable $G^{i,j}\sim \mathcal{N}(0,\sqrt{\Delta_t}I_d)$.
        \STATE Choose an initial control $(\alpha^{i}_j)_{(i,j)\in \llbracket 1, N_p\rrbracket\times \llbracket 1, N_t\rrbracket}$ progressively adapted. 
        \STATE set $\bar{\alpha}=\alpha$.
        \STATE $\varepsilon_{error}=\|v^{N^t,N_p}((X_0^i,G^{i}_j)_{(i,j)\in \llbracket 1, N_p\rrbracket\times \llbracket 1, N_t\rrbracket})(\bar{\alpha})\|_T$
        \STATE Initialize $Y=\frac{1}{\gamma_1}\alpha$
        \WHILE{$\varepsilon_{error}>\varepsilon_0$}
            \STATE $N_{iter}\leftarrow N_{iter}+1$
            \STATE $\alpha\leftarrow \alpha -\gamma_{N_{iter}} v^{N_t,N_p}((X_0^i,G^{i,j})_{(i,j)\in \llbracket 1, N_p\rrbracket\times \llbracket 1, N_t\rrbracket})(\alpha)$
            \STATE $\bar{\alpha}\leftarrow (1-\frac{1}{N_{iter}+1})\bar{\alpha}+\frac{1}{N_{iter}+1}\alpha$
            \STATE $Y\leftarrow Y-v^{N_t,N_p}((X_0^i,G^{i}_j)_{(i,j)\in \llbracket 1, N_p\rrbracket\times \llbracket 1, N_t\rrbracket})(\alpha)$
            \STATE $\alpha\leftarrow \gamma_{N_{iter}} Y$
            \STATE $\varepsilon_{error}\leftarrow \|v^{N^t,N^t}((X_0^i,G^{i}_j)_{(i,j)\in \llbracket 1, N_p\rrbracket\times \llbracket 1, N_t\rrbracket})(\bar{\alpha})\|_T$
        \ENDWHILE
    \STATE Return $\bar{\alpha}$.
    \end{algorithmic}
\end{algorithm}

Where we used the notation 
\[\forall \alpha \in \reels^{N_p\times N_t\times d}\quad \| \alpha\|_T= \sqrt{\frac{\Delta_t}{N_p}\sum_{i,j} \|\alpha^{i}_j\|^2},\]
which is the corresponding $L^2-$norm for discretized random processes. 
\subsubsection{On the convergence rate of the particle method}
We now present some insight on the expected convergence rate of this algorithm to the solution of \eqref{FBSDE}. To that end we first introduce the following particle system 
\begin{lemma}
Under Hypotheses \ref{hyp: lipschitz fbsde} and \ref{hyp: L2 monotonicity}, for any $T>0, N>1$ and for any admissible initial condition $(X_0^1,\cdots X_0^N)$, there exists a unique strong solution to the system 
\begin{equation}
    \label{eq: N particle system}
\left\{
\begin{array}{l}
     \displaystyle X^{i,N}_t=X^i_0-\int_0^t  F(X^{i,N}_s,U^{i,N}_s,m^N_t)ds+\sqrt{2\sigma}B^i_t, \\
    \displaystyle  U^{i,N}_t= U^{i,N}_T+\int_t^T G(X^{i,N}_s,U^{i,N}_s,m^N_s)ds-\int_t^T Z^{i,N}_sd(B^1_s,\cdots B^N_s),\\
    U^{i,N}_T=g(X^{i,N}_T,\frac{1}{N} \sum_j\delta_{X^{j,N}_T}),\\
    m^N_t=\frac{1}{N}\sum_j \delta_{X^{j,N}_t,U^{j,N}_t}.
\end{array}
\right.
\end{equation}
Where $(B^1_t,\cdots B^N_t)_{t\geq 0}$ is a $Nd-$dimensional Brownian motion.
\end{lemma}
\begin{proof}
For any $\mathbf{x},\mathbf{u}\in \reels^{Nd}$ 
\[\mathbf{x}=(x^1,\cdots, x^n), \quad x^i\in \reels^d \quad \forall i\in \llbracket 1,N\rrbracket,\]
we introduce the notation 
\[\mathbf{F}(\mathbf{x},\mathbf{u})=\left(F\left(x^i,u^i,\frac{1}{N}\sum_j \delta_{x^j,u^j}\right)\right)_{i\in \llbracket 1,N\rrbracket},\]
and an analogous notation for $g$ and $G$. Similarly for stochastic processes we denotes 
\[\mathbf{X}^N_t=(X^{1,N}_t,\cdots X^{N,N}_t),\]
and so on. It follows that the sytem \eqref{eq: N particle system} is equivalent to 
\begin{equation*}
\left\{
\begin{array}{l}
     \displaystyle \mathbf{X}^N_t=\mathbf{X}_0-\int_0^t  \mathbf{F}(\mathbf{X}^N_s,\mathbf{U}^N_s)ds+\sqrt{2\sigma}\mathbf{B}_t, \\
    \displaystyle  \mathbf{U}^N_t= \mathbf{g}(\mathbf{X}^N_T)+\int_t^T \mathbf{G}(\mathbf{X}^N_s,\mathbf{U}^N_s)ds-\int_t^T \mathbf{Z}_sd\mathbf{B}_s.\\
\end{array}
\right.
\end{equation*}
Moreover, by the regularity and monotonicity of $F,G,g$, $\mathbf{F},\mathbf{G},\mathbf{g}$ are Lipschitz and 
\begin{gather*}
\forall \mathbf{x},\mathbf{y},\mathbf{u},\mathbf{v}\in \reels^{Nd},\\
(\mathbf{G}(\mathbf{x},\mathbf{u})-\mathbf{G}(\mathbf{y},\mathbf{v}))\cdot(\mathbf{x}-\mathbf{y})+(\mathbf{F}(\mathbf{x},\mathbf{u})-\mathbf{F}(\mathbf{y},\mathbf{v}))\cdot(\mathbf{u}-\mathbf{v})\geq c_F |\mathbf{u}-\mathbf{v}|^2,\\
\mathbf{g}(\mathbf{x})-\mathbf{g}(\mathbf{y})\cdot(\mathbf{x}-\mathbf{y})\geq 0.
\end{gather*}
Since coefficients are Lipschitz and monotone, the existence and uniqueness of a strong solution are a consequence of Theorem \ref{thm: existence fbsde}.
\end{proof}
Indicating by $\|\cdot \|_{k,p}$ the norm on the Sobolev space $W^{k,p}(\reels^{2d})$ and by $\|\cdot\|_{-k,p}$ the associated negative Sobolev norm, we can now show the following theorem estimating the error of our algorithm
\begin{thm}
    \label{thm: approximation fbsde}
Under Hypotheses \ref{hyp: lipschitz fbsde} and \ref{hyp: L2 monotonicity}, fix $N\geq 1$ and let $(X_0^1,\cdots X_0^N)$ an admissible initial condition, such that the $(X_0^i)_{i\in \llbracket 1,N\rrbracket }$ are iid. We fix $\mathbf{\alpha}=(\alpha^1,\cdots \alpha^N)\in \left(\mathcal{H}^T\right)^{Nd}$ a control and consider the dynamics 
\begin{equation}
    \label{eq: N particle controlled system}
\left\{
\begin{array}{l}
     \displaystyle X^{\alpha,i}_t=X^i_0-\int_0^t  \alpha^i_sds+\sqrt{2\sigma}B^i_t, \\
    \displaystyle  U^{\alpha,i}_t= U^\alpha_T+\int_t^T G(X^{\alpha,i}_s,F_{inv}(X^{\alpha,i}_s,\alpha^i_s,m^\alpha_s),\nu^\alpha_s)ds-\int_t^T Z^{\alpha,i}_sd(B^1_s,\cdots B^N_s),\\
    U^\alpha_T=g(X^{\alpha,i}_T,\frac{1}{N} \sum_j\delta_{X^{\alpha,j}_T}),\\
    m^\alpha_t=\frac{1}{N}\sum_j \delta_{(X^{\alpha,j}_t,\alpha^i_t)},\\
    \nu^\alpha_t=\frac{1}{N}\sum_j \delta_{(X^{\alpha,j}_t,F_{inv}(X^{\alpha,j}_t,\alpha^j_t,m^\alpha_t))}.
\end{array}
\right.
\end{equation}
 Letting 
    \[\forall t\in [0,T],\quad v(\alpha^i)_t=F_{inv}(X^{\alpha,i}_t,\alpha^i_t,m^\alpha_t)-U^{\alpha,i}_t,\]
    and $(X^i_t,U^i_t,Z^i_t)_{t\in [0,T]}$ be the unique strong solution of 
\begin{equation*}
\left\{
\begin{array}{l}
     \displaystyle X^{i}_t=X^i_0-\int_0^t  F(X^{i}_s,U^{i}_s,m^i_t)ds+\sqrt{2\sigma}B^i_t, \\
    \displaystyle  U^{i}_t= g(X^{i}_T,\mathcal{L}(X^i_T))+\int_t^T G(X^{i}_s,U^{i}_s,m^i_s)ds-\int_t^T Z^{i}_sdB^i_s,\\
    m^i_t=\mathcal{L}(X^i_t,U^i_t).
\end{array}
\right.
\end{equation*}
    there exists a constant $C_{coef}$ depending only on $c_F,T,\|F,G,g\|_{Lip}$ such that 
    \begin{enumerate}
        \item[-] If for some $q>2$ 
        \[\esp{|X_0^1|^q}<+\infty,\]
        then there exists a constant $C_{fg}$ depending on $q,d$ only such that for any $i\in \llbracket 1,N\rrbracket$
\begin{gather*} \esp{\sup_{[0,T]}|X^{\alpha,i}_t-X^{i}_t|^2+ \sup_{[0,T]}|U^{\alpha,i}_t-U^{i}_t|^2}\\\leq C_{coef}\left(\|v(\alpha^i)\|_T^2+\frac{1}{N}\sum_j \|v(\alpha^j)\|_T^2+C_{fg}\left(\esp{|X_0^1|^q}\right)^\frac{2}{q}\mathcal{E}(N,q,d)\right),\end{gather*}
with 
\[\mathcal{E}(N,q,d)=\left\{
\begin{array}{l}
N^{-1/2}+N^{-(q-2)/q} \text{ if } d<2 \text{ and } q\neq 4\\
N^{-1/2}\text{log}(1+N)+N^{-(q-2)/q} \text{ if } d=2 \text{ and } q\neq 4\\
N^{-1/d}+N^{-(q-2)/q} \text{ if } d>2 \text{ and } q\neq \frac{d}{d-1}   
\end{array}
\right.\]
\item[-] If $F,G,g$ admit derivatives with respect to the measure argument up to the second order such that for $f=(F,G,g)$
\[\|f\|_{Lip},\|\frac{\delta f}{\delta m}\|_{Lip}\|\frac{\delta^2 f}{\delta m^2}\|_{Lip}\leq C_{Lip,2},\]
then there exists a constant $C_{smooth}$ depending on $C_{Lip,2},d,T$ and $m=\mathcal{L}(X_0^1)$ such that 
\begin{gather*}\esp{\sup_{[0,T]}|X^{\alpha,i}_t-X^{i}_t|^2+ \sup_{[0,T]}|U^{\alpha,i}_t-U^{i}_t|^2}\\\leq C_{coef}\left(\|v(\alpha^i)\|_T^2+\frac{1}{N}\sum_j \|v(\alpha^j)\|_T^2+\frac{C_{smooth}}{N}\right),\end{gather*}
\item[-] If $F,G,g$ admit a first order derivative with respect to the measure argument and there exists a constant $C_{reg}$, such that for $f=F,G,g$
\[\forall (x,m)\in \reels^{2d}\times \mathcal{P}_2(\reels^{2d}), \quad \sup_{\begin{array}{c} x\in \reels^{2d}\\ m\in \mathcal{P}_2(\reels^{2d})\end{array}}\|\frac{\delta f}{\delta m}(x,m,\cdot)\|_{(2+d,2(d+1))}\leq C_{reg},\]
and \[\esp{|X_0^1|^{8(d+1)}}<C_{reg},\]
then there exists a constant $C_{sobolev}$ depending on $C_{reg}$, $\|(F,G,g)\|_{Lip},T,c_F$ and $d$ such that 
\begin{gather*}\esp{\sup_{[0,T]}|X^{\alpha,i}_t-X^{i}_t|^2+ \sup_{[0,T]}|U^{\alpha,i}_t-U^{i}_t|^2}\\
    \leq C_{coef}\left(\|v(\alpha^i)\|_T^2+\frac{1}{N}\sum_j \|v(\alpha^j)\|_T^2+\frac{C_{sobolev}}{N}\right),\end{gather*}
\end{enumerate}

\end{thm}
\begin{proof}
    \quad

\flushleft\textit{Step 1: a stability estimate}\\
    By definition of $F_{inv}$ 
    \[\alpha^i_t= F(X^{\alpha,i}_t,U^{\alpha,i}_t+v(\alpha^i)_t,\nu^\alpha_t).\]
    Moreover $(\nu^\alpha_t)_{t\in [0,T]}$ can be rewritten in function of $(v(\alpha^j)_t)_{(t,j)\in [0,T]\times \llbracket 1,N\rrbracket}$ and $(U^{\alpha,j}_t)_{(t,j)\in [0,T]\times \llbracket 1,N\rrbracket}$ as 
    \[\forall t\in [0,T], \quad \nu^\alpha_t=\frac{1}{N}\sum_j \delta_{(X^{\alpha,j}_t,U^{\alpha,j}_t+v(\alpha^j)_t)}.\]
    Consequently the system \eqref{eq: N particle controlled system} is equivalent to 
    \begin{equation*}
\left\{
\begin{array}{l}
     \displaystyle X^{\alpha,i}_t=X^i_0-\int_0^t  F(X^{\alpha,i}_s,U^{\alpha,i}_s+v(\alpha^i)_s,\nu^\alpha_s)ds+\sqrt{2\sigma}B^i_t, \\
    \displaystyle  U^{\alpha,i}_t= U^\alpha_T+\int_t^T G(X^{\alpha,i}_s,U^\alpha_s+v(\alpha^i)_s,\nu^\alpha_s)ds-\int_t^T Z^{\alpha,i}_sd(B^1_s,\cdots B^N_s),\\
    U^\alpha_T=g(X^{\alpha,i}_T,\frac{1}{N} \sum_j\delta_{X^{\alpha,j}_T}),\\
    \nu^\alpha_t=\frac{1}{N}\sum_j \delta_{(X^{\alpha,j}_t,U^{\alpha,j}_t+v(\alpha^j)_t)}.
\end{array}
\right.
\end{equation*}
By a direct variation of Lemma \ref{lemma: error in function of v(alpha)} applied to the system \eqref{eq: N particle system}, we deduce that there exists a constant $C_1$ such that 
\[\esp{\sum_j \sup_{[0,T]}|X^{\alpha,j}_t-X^{j,N}_t|^2+\sum_j \sup_{[0,T]}|U^{\alpha,j}_t-U^{j,N}_t|^2}\leq C_1\sum_j \|v(\alpha^j)\|_T^2.\]
Then, in light of this estimate and by applying the same stability lemma to compare only two particle we deduce that for any $i\in \llbracket 1,N\rrbracket$
\[\esp{\sup_{[0,T]}|X^{\alpha,i}_t-X^{i,N}_t|^2+ \sup_{[0,T]}|U^{\alpha,i}_t-U^{i,N}_t|^2}\leq C_1\left(\|v(\alpha^i)\|_T^2+\frac{1}{N}\sum_j \|v(\alpha^j)\|_T^2\right).\]
We now introduce the following FBSDE system 
\begin{equation}
    \label{eq: N particle limit system}
\left\{
\begin{array}{l}
     \displaystyle X^{i}_t=X^i_0-\int_0^t  F(X^{i}_s,U^{i}_s,m^i_t)ds+\sqrt{2\sigma}B^i_t, \\
    \displaystyle  U^{i}_t= g(X^{i}_T,\mathcal{L}(X^i_T))+\int_t^T G(X^{i}_s,U^{i}_s,m^i_s)ds-\int_t^T Z^{i}_sdB^i_s,\\
    m^i_t=\mathcal{L}(X^i_t,U^i_t).
\end{array}
\right.
\end{equation}
The wellposedness of this system for any admissible condition has already been established in Theorem \ref{thm: convergence fbsde}. In particular 
\[\forall i,j\in \llbracket 1,N\rrbracket^2, \quad \forall t\in [0,T],\quad \mathcal{W}_2(m^i_t,m^j_t)=0, \]
and we may drop the indice by considering only $(m_t)_{t\in [0,T]}=(m^1_t)_{t\in [0,T]}$. Introducing 
\[\forall t\in [0,T], \quad \mu^N_t=\frac{1}{N}\sum_j \delta_{X^j_t,U^j_t},\]
and 
\[
\forall t\in [0,T],\quad  \left\{\begin{array}{l}
    \eta^{x,i}_t=F(X^i_t,U^i_t,m_t)-F(X^i_t,U^i_t,\mu^N_t),\\
    \eta^{u,i}_t=G(X^i_t,U^i_t,m_t)-G(X^i_t,U^i_t,\mu^N_t),\\
    \gamma^i_T=g(X^i_T,\mathcal{L}(X^1_T))-g(X^i_T,\frac{1}{N}\sum_j \delta_{X^j_T}),
\end{array}\right.\]
We deduce by Remark \ref{remarque: strong stability of fbsde} that there exists a constant $C_2$ such that
\begin{gather*}\esp{\sum_j \sup_{[0,T]}|X^{j}_t-X^{j,N}_t|^2+\sum_j \sup_{[0,T]}|U^{j}_t-U^{j,N}_t|^2}\\\leq C_2\sum_j \left(\underbrace{\|\gamma_T^j\|^2+\int_0^T \|(\eta^{x,j}_t,\eta^{u,j}_t)\|^2dt}_{\varepsilon_j}\right).\end{gather*}
Then using this newfound bound to estimate the difference between two particles we find 
\[\forall i\in \llbracket 1,N\rrbracket, \quad \esp{ \sup_{[0,T]}|X^{i}_t-X^{i,N}_t|^2+\sup_{[0,T]}|U^{i}_t-U^{i,N}_t|^2}\leq C_2\left(\varepsilon_i+\frac{1}{N}\sum_j \varepsilon_j\right).\]
Putting together those estimates by using Cauchy-Schwartz inequality finaly yields that for any $i\in \llbracket 1,N\rrbracket,$
\begin{gather}
    \label{eq: estimating difference particles}
    \esp{\sup_{[0,T]}|X^{\alpha,i}_t-X^{i}_t|^2+ \sup_{[0,T]}|U^{\alpha,i}_t-U^{i}_t|^2}\\
    \nonumber \leq C\left(\|v(\alpha^i)\|_T^2+\varepsilon_i+\frac{1}{N}\sum_j \left(\|v(\alpha^j)\|_T^2+\varepsilon_j\right)\right),
\end{gather}
where the constant $C$ is given by following the proof of Lemma \ref{lemma: error in function of v(alpha)}. In particular it depends only on $T,c_F, \|G,F,g\|_{Lip}$. 

\flushleft\textit{Step 2: bounding the error terms in function of N}\\

At this point, the first statement follows directly by noticing that for any $i \in \llbracket 1,N\rrbracket$
\[ \varepsilon_i\leq C_{Lip}\left(\esp{\mathcal{W}_2^2(\frac{1}{N}\sum_j \delta_{X^j_T},\mathcal{L}(X^1_T))}+\int_0^T\esp{\mathcal{W}_2^2(\mu^N_t,\mathcal{L}(X^1_t,U^1_t))}dt\right),\]
where $C_{Lip}$ depends only on the Lipschitz constant of $F,G,g$ and applying the sharp estimate \cite{Fournieretal2014} Theorem 1 for the Wasserstein$-2$ distance and in dimension $2d$. The second and third statement are a consequence of the fact that whenever coefficients are sufficiently smooth, the errors terms $(\varepsilon_j)_{j\in \llbracket 1,N\rrbracket }$ can be estimated with a better rate in function of $N$. Let us fix $(t,i)\in [0,T]\times \llbracket 1,N\rrbracket$, we remind that 
\[\eta^{x,i}_t=F(X^i_t,U^i_t,m_t)-F(X^i_t,U^i_t,\mu^N_t).\]
Introducing 
\[\mu^i_t=\frac{1}{N-1}\sum_{j\neq i} \delta_{(X^j_t,U^j_t)},\]
we deduce that 
\[|\eta^{x,i}_t|^2\leq 2\left(F(X^i_t,U^i_t,m_t)-F(X^i_t,U^i_t,\mu^i_t)\right)^2+2\left(F(X^i_t,U^i_t,\mu^N_t)-F(X^i_t,U^i_t,\mu^i_t)\right)^2.\]
Moreover since $F$ is Lipschitz in $\mathcal{P}_2(\reels^{2d})$, 
\[\left(F(X^i_t,U^i_t,\mu^N_t)-F(X^i_t,U^i_t,\mu^i_t)\right)^2\leq \|F\|^2_{Lip}\mathcal{W}_2^2(\mu^N_t,\mu^i_t), \quad a.s,\]
and by definition of the Wasserstein$-2$ distance 
\[\mathcal{W}_2^2(\mu^N_t,\mu^i_t)\leq \frac{1}{N(N-1)}\sum_{j\neq i}\left|(X^i_t,U^i_t)-(X^j_t,U^j_t)\right|^2, \quad a.s.\]
Consequently 
\[\esp{\left(F(X^i_t,U^i_t,\mu^N_t)-F(X^i_t,U^i_t,\mu^i_t)\right)^2}\leq \frac{2}{N}\esp{|(X^1_t,U^1_t)|^2}.\]
It remains to show that 
\[\esp{\left(F(X^i_t,U^i_t,m_t)-F(X^i_t,U^i_t,\mu^i_t)\right)^2},\]
can be bounded with rate $\frac{1}{N}$. To that end we first notice that $\mu^i_t$ is independent of $(X^i_t,U^i_t)$, as a consequence it is sufficient to show that we can bound
\[\esp{\left(F(x,u,m_t)-F(x,u,\mu^i_t)\right)^2},\]
by a factor $\frac{1}{N}$ uniformly in $(x,u)\in \reels^{2d}$. The second statement is then a direct application of \cite{Jourdainetal2021} Corollary 3.3. As for the third statement fixing $(x,u)\in \reels^{2d}$, if $F$ has a continuous first order derivative then 
\[F(x,u,m_t)-F(x,u,\mu^i_t)=\int_0^1 \int_{\reels^{2d}}\frac{\delta F}{\delta m}(x,u,\lambda m_t+(1-\lambda)\mu^i_t,y)(m_t-\mu^i_t)(dy)d\lambda,\quad a.s.\]
By definition of negative sobolev norms, under the assumptions for the third statement, 
\[|F(x,u,m_t)-F(x,u,\mu^i_t)|\leq C_{reg}\|m_t-\mu^i_t\|_{-(2+d),2(d+1)},\quad a.s.\]
It remains to show that in expectation this norm can be estimated with rate $\frac{1}{N}$. If 
\[\esp{|(X^1_t,U^1_t)|^{8(d+1)}}<+\infty,\]
then this result can be found in the proof of Proposition 3.5 of \cite{FernandezMeleard1997}. Since by assumption 
\[\esp{|X^1_0|^{8(d+1)}}<+\infty,\]
this is a consequence of Gronwall Lemma. Indeed by \cite{monotone-sol-meynard} Theorem 3.33 (and its variant in the presence of idiosyncratic noise) there exists a Lipschitz function $U:[0,T]\times \reels^d\times \mathcal{P}_2(\reels^d)\to \reels^d$ such that 
\[\forall t\in [0,T], \quad U^1_t=U(t,X^1_t,\mathcal{L}(X^1_t)).\]
This allows to decouple the forward backward system and conclude with a standard Gronwall type argument. 
\end{proof}
\begin{remarque}
The conditions we give to ensure that the convergence occurs with rate $\frac{1}{\sqrt{N}}$ are not optimal. Rather we just want to stress that with enough regularity with respect to the measure argument it is possible to get back a convergence rate depending mildly (only through constants) on the dimension of the processes. This is especially relevant for Forward-Backward systems since we have to work with measures in $\mathcal{P}(\reels^{2d})$ as soon as the forward process lives in $\reels^d$. 
\end{remarque}
In particular let us insist once again that in the case of MFGs of controls $F_{inv}$ is given directly by $\nabla_\alpha L$, this leads directly to the following corollary 
\begin{corol}
Under Hypotheses \ref{disp monotone mfg} and \ref{hyp: lipschitz regularity mfgc}, fix $N\geq 1$ and let $(X_0^1,\cdots X_0^N)$ be an admissible initial condition, such that the $(X_0^i)_{i\in \llbracket 1,N\rrbracket }$ are iid. We fix $\mathbf{\alpha}=(\alpha^1,\cdots \alpha^N)\in \left(\mathcal{H}^T\right)^{Nd}$ a control and consider the dynamics 
\begin{equation}
\left\{
\begin{array}{l}
     \displaystyle X^{\alpha,i}_t=X^i_0-\int_0^t  \alpha^i_sds+\sqrt{2\sigma}B^i_t, \\
    \displaystyle  U^{\alpha,i}_t= U^\alpha_T+\int_t^T \nabla_x L(X^{\alpha,i}_s,\alpha^i_s,m^\alpha_s)ds-\int_t^T Z^{\alpha,i}_sd(B^1_s,\cdots B^N_s),\\
    U^\alpha_T=g(X^{\alpha,i}_T,\frac{1}{N} \sum_j\delta_{X^{\alpha,j}_T}),\\
    m^\alpha_t=\frac{1}{N}\sum_j \delta_{(X^{\alpha,j}_t,\alpha^i_t)},
\end{array}
\right.
\end{equation} 
    \[\forall t\in [0,T],\quad v(\alpha^i)_t=\nabla_\alpha L(X^{\alpha,i}_t,\alpha^i_t,m^\alpha_t)-U^{\alpha,i}_t,\]
     and $(\alpha^{*,i}_t)_{t\in [0,T]}$ be the unique solution to the mean field game \eqref{mfgc} with initial condition $X^i_0$.
    There exists a constant $C_{coef}$ depending only on $c_L,T,\|\nabla_x L,\nabla_\alpha L,\nabla_x g\|_{Lip}$ such that if for some $q>2$ 
        \[\esp{|X_0^1|^q}<+\infty,\]
        then there exists a constant $C_{fg}$ depending on $q,d$ only such that for any $i\in \llbracket 1,N\rrbracket$
\begin{gather*} \esp{\sup_{[0,T]}|X^{\alpha,i}_t-X^{\alpha^{*,i}}_t|^2+ \sup_{[0,T]}|\alpha^{*,i}_t-\alpha^{i}_t|^2}\\\leq C_{coef}\left(\|v(\alpha^i)\|_T^2+\frac{1}{N}\sum_j \|v(\alpha^j)\|_T^2+C_{fg}\left(\esp{|X_0^1|^q}\right)^\frac{2}{q}\mathcal{E}(N,q,d)\right),\end{gather*}
\end{corol}
\begin{remarque}
For MFGs of controls this yields the estimate 
\[\esp{\mathcal{W}^2_2(\frac{1}{N}\sum_j \delta_{(X^{\alpha,j}_t,\alpha^i_t)},\mathcal{L}(X^{\alpha^{*,1}}_t,\alpha^{*,1}_t))}\leq C\left(\frac{1}{N}\sum_i \|v(\alpha^i)\|^2_T+\mathcal{E}(N,q,d)\right). \]
Let us insist that up to the term $\sqrt{\frac{1}{N}\sum_i \|v(\alpha^i)\|^2_T}$, which converges exponentially fast to 0 with the number of iterations of our proposed scheme, we get the same convergence rate as one would for a particle scheme for a standard McKean-Vlasov SDE without any backward component. That is, the term $\mathcal{E}(N,q,d)$ which is of the order $\frac{1}{N^\frac{1}{d}}$ comes from the particle approximation of the measure argument and is not avoidable in general. If coefficients are sufficiently smooth with respect to the measure argument, it is always possible to show that the convergence rate is in fact of the order 
\[\esp{\mathcal{W}^2_2(\frac{1}{N}\sum_j \delta_{(X^{\alpha,j}_t,\alpha^i_t)},\mathcal{L}(X^{\alpha^{*,1}}_t,\alpha^{*,1}_t))}\leq C\left(\frac{1}{N}\sum_i \|v(\alpha^i)\|^2_T+\frac{1}{N}\right), \]
as we demonstrated in Theorem \ref{thm: approximation fbsde}.
\end{remarque}
Let us end this paragraph by adressing the additional error brought by the time discretisation of the problem. First, since the forward process evolves linearly for a given control $(\alpha^i_t)_{(t,i)\in [0,T]\times \llbracket 1,N\rrbracket}$, the scheme we propose is exact forward in time if we see the class of discrete controls we introduced as processes in continuous time but piecewise constant. Then using the formulation of the backward process as a conditional expectation and since coefficients do not depend on martingale part of the BSDE, it is straightforward to obtain that this discretisation is of rate $\Delta_t$. In the best case scenario (that is when coefficients are sufficiently regular with respect to the measure argument), we expect that for a discrete control $(\alpha^j_k)_{j,k\in \llbracket 1,N\rrbracket \times \llbracket 1,N_t\rrbracket}$ the error satisfies
\begin{gather*} 
    \esp{\sup_{k\in \llbracket 1,N_t\rrbracket}|X^{\alpha,i}_k-X^{i}_{k\Delta_t}|^2+\sup_{k\in \llbracket 1,N_t\rrbracket}|U^{\alpha,i}_k-U^{i}_{k\Delta_t}|^2}\\\leq C_{error}\left(\Delta_t^2+\sum_{k}\Delta_t|v(\alpha^i)_k|^2+\frac{1}{N}\sum_j\sum_k \Delta_t|v(\alpha^j)_k|^2+\frac{1}{N}\right),
\end{gather*}
in which case $\Delta_t$ should be chosen such that 
\[N_t\sim \sqrt{N}.\]
Even under this discretisation, we expect that the sequence $(\alpha^n)_{n\in \mathbb{N}}\subset\left(\reels^{N_p\times N_t\times d} \right)^\mathbb{N}$ introduced in Algorithm \ref{algo 1} converge exponentially fast toward $0$ in the space of discrete controls (in practise this is what we observe, and the coefficients associated to the discretised FBSDE in time are still monotone). However since this has no impact on the error coming from the space and time discretisation, the algorithm should be stopped once 
\[\|\alpha^n\|^2_T=\frac{1}{N}\sum_j\sum_k \Delta_t|v(\alpha^j)_k|^2\sim \frac{1}{\sqrt{N_t}}\sim \frac{1}{N}.\]
\subsubsection{FBSDEs with common noise}
Let us briefly explain the main differences in this setting. Since we must be able to compute an approximation of the conditional law at each backward step, this implies that the generation of the discrete idiosyncratic noise $(G^{i}_j)_{i=1:N_p,j=1:N_t}$ and the discrete common noise $(Z^{k}_j)_{k=1:N_0,j=1:N_t}$ must be done independently (where we choose to simulate $N_0$ trajectories of the common noise). A discrete control is now also a collection $(\alpha^{i,k}_j)_{i=1:N_p,j=1:N_t,k=1:N_0}$ which depends on the trajectories of common noise. We now present an algorithm in pseudocode to compute the function 
\[v^{N_t,N_p,N_0}:\reels^{N_p\times N_t\times N_0\times d}
\to\reels^{N_p\times N_t\times N_0\times d},\]
discretization of the associated functional \eqref{eq: v with common noise}
\begin{algorithm}[H]
    \caption{Computation of $v^{N_t,N_p,N_0}$}\label{algo v common}
    \begin{algorithmic}
        \STATE INPUT: $(\alpha^{i,k}_j)_{(i,j,k)\in \llbracket 1, N_p\rrbracket\times \llbracket 1, N_t\rrbracket\times \llbracket 1, N_0\rrbracket}$
        \STATE -Compute the associated forward processes $(X^{i,k}_j,p^{k}_j)_{(i,j,k)\in \llbracket 1, N_p\rrbracket\times \llbracket 1, N_t\rrbracket\times \llbracket 1, N_0\rrbracket}$ iteratively through 
        \[\left\{\begin{array}{l}
            X^{i,k}_{j+1}=X^{i,k}_j+\Delta_t \alpha^{i,k}_j+\sqrt{2\sigma}G^{i}_j,\\
        p^{k}_{j+1}=p^{k}_j+\Delta_t b(p^{k}_j)+\sqrt{2\sigma^0}Z^{k}_j
        \end{array}\right..\]
        \STATE -The backward process $(U^{i,k}_j)_{(i,j,k)\in \llbracket 1, N_p\rrbracket\times \llbracket 1, N_t\rrbracket\times \llbracket 1, N_0\rrbracket}$ is computed recursively as follows: given $(U^{i+1,k}_j)_{(i,j)\in \llbracket 1, N_p\rrbracket\times \llbracket 1, N_t\rrbracket}$, $(U^{i,k}_j)_{(i,j)\in \llbracket 1, N_p\rrbracket\times \llbracket 1, N_t\rrbracket}$ is computed by regression of 
        \[U^{i,k}_{j+1}+\Delta_t G(X^{i,k}_{j+1},p^{k}_{j+1},F_{inv}\left(X^{i,k}_{j+1},p^{k}_{j+1},\alpha^{i,k}_{j+1},m_{j+1}^k\right),\nu^k_{j+1})\]
        into $(X^{i,k}_j,p^{k}_j)_{i=1:N_p,k=1:N_0}$, with the empirical measures 
        \[\left\{\begin{array}{l}m_j^k=\frac{1}{N_p}\sum_l\delta_{(X^{l,k}_j,\alpha^{l,k}_j)},\\
        \nu_j^k=\frac{1}{N_p}\sum_l\delta_{\left(X^{l,k}_j,F_{inv}\left(X^{i,k}_{j},p^{k}_{j},\alpha^{i,k}_{j},m_{j}^k\right)\right)},
        \end{array}\right.\]
    \STATE -Return $\left(F_{inv}\left(X^{i,k}_j,p^{k}_j,\alpha^{i,k}_j,m_j^k\right)-U^{i,k}_j\right)_{(i,j,k)\in \llbracket 1, N_p\rrbracket\times \llbracket 1, N_t\rrbracket\times \llbracket 1, N_0\rrbracket}$.
    \end{algorithmic}
\end{algorithm}
Clearly this algorithm is much more costly compared to the version without common noise. First the memory complexity is now of the order of $O(N_t\times N_p\times N_0)$. Moreover at each backward time step, we compute the regression of a function in $\reels^d\times \reels^{d^0}$
\subsection{Illustration on a FBSDE}
Compared to the deterministic case, the most time consuming step consists in approximating the backward process. Although the method we propose allows to iterate on decoupled FBSDEs, we still need to estimate a conditional expectation at each time step. In the following test, it is estimated with polynomial regression on the first $10$ Hermite polynomials. 
\subsubsection{System considered}
We consider the following exemple 
\begin{equation}
\label{eq: explicit mffbsde numerical}
\left\{
    \begin{array}{c}
        \displaystyle X_t=X_0-\int_0^t aY_sds+\sqrt{2\sigma}B_t,\\
       \displaystyle  Y_t=bX_T+\int_t^T \left(cX_s+\esp{f(X_s-\esp{X_s})}\right)ds-Z_sdB_s,
    \end{array}\right.
\end{equation}
In this case the solution is known explicitely. Indeed Let us first assume that $(Y_s)_{s\in [0,T]}$ can be writen in feedback form as 
\begin{equation}
 \label{eq: backward process in feedback for exemple}   
    \forall t\in [0,T], \quad Y_t=\eta(t)X_t+\theta(t),
\end{equation}
for smooth functions of time $\eta,\theta$. It follows directly that 
\[\bar{X}_t=X_t-\esp{X_t},\]
is given by 
\[\displaystyle \bar{X}_t=\bar{X}_0e^{-a\int_0^t \eta(s)ds }+\sqrt{2\sigma^x}\int_0^t e^{-\int_s^t a\eta(u)du}dW_s.\]
In particular, there exists a continuous function $e:[0,T]\to \reels$ depending on $\eta,a,\mathcal{L}(X_0)$ only such that 
\[\forall t\in [0,T],\quad \esp{f(X_t-\esp{X_t})}=\esp{f(\bar{X}_t)}=e(t).\]
Using the representation formula \eqref{eq: backward process in feedback for exemple}, we deduce that 
\begin{align*}\forall t\in [0,T], \quad 0=&\int_t^T\left((c+\frac{d\eta}{ds}(s)-a \eta^2(s))X_s+e(s)-a\eta(s)\theta(s)+\frac{d\theta}{ds}(s)\right)ds\\
    &+\int_t^T(\sqrt{2\sigma}\eta(s)-Z_s)dW_s.\end{align*}
In the end it suffices to solve the following system 
\[
\left\{
    \begin{array}{l}
        c+\frac{d\eta}{ds}(s)-a \eta^2(s)=0, \quad \eta(T)=b,\\
        e(s)-a\eta(s)\theta(s)+\frac{d\theta}{ds}(s)=0, \quad \theta(T)=0.
    \end{array}
\right.
\]
In particular if $a,b,c\geq 0$ and $f$ is either monotone or Lipschitz with $\|f\|_{Lip}\leq c$, we are in the monotone regime and the system admits a unique solution on any interval of time. Since monotonicity also implies uniqueness of a solution to \eqref{eq: explicit mffbsde numerical}, we have exhibited the unique solution to this mean field FBSDE. For $c=0$
\[
\left\{
    \begin{array}{l}
        \displaystyle \eta(t)=\frac{b}{1+ab (T-t)},\\
        \displaystyle \bar{X}_t=\bar{X}_0\left(1-\frac{ab t}{1+ab T}\right)+\sqrt{2\sigma}(1+ab(T-t))\int_0^t  \frac{1}{1+ab(T-s)}dW_s,\\
        \displaystyle \theta(t)=\frac{1}{1+ab(T-t)}\int_{T-t}^Te(s)(1+ab(T-s))ds\\
        e(t)=\esp{f(\bar{X}_t)}.
    \end{array}
\right.
\]
Choosing $X_0=x\in \reels$, 
\[\bar{X}_t\sim \mathcal{N}\left(0,\frac{2\sigma}{ab}\left(1+ab(T-t)-\frac{(1+ab(T-t))^2}{1+ab T}\right)\right),\]
and $e(t)$ can be computed very accurately with gaussian quadrature. 
\subsubsection{Numerical results}
We now present some numerical result for the system \eqref{eq: explicit mffbsde numerical} for the parameters 
\[
\left\{
\begin{array}{l}
    \sigma=1,\\
    a=1,\\
    b=1,\\
    x_0=1,\\
    c=0,\\
    T=10,
\end{array}
\right. ,
\]
with a step size $\gamma=0.08$ and 
\[\forall x\in \reels, \quad f(x)=atan(x-1).\]
Following the algorithm \ref{algo FBSDE} with $N_t=100,N_p=10000$ and letting 
\[\varepsilon_{error}(n)=\|v^{N_t,N_p}(\alpha_n)\|_T,\]
the sequence of errors on the last iterate, we plot its logarithm in function of the number of iterations
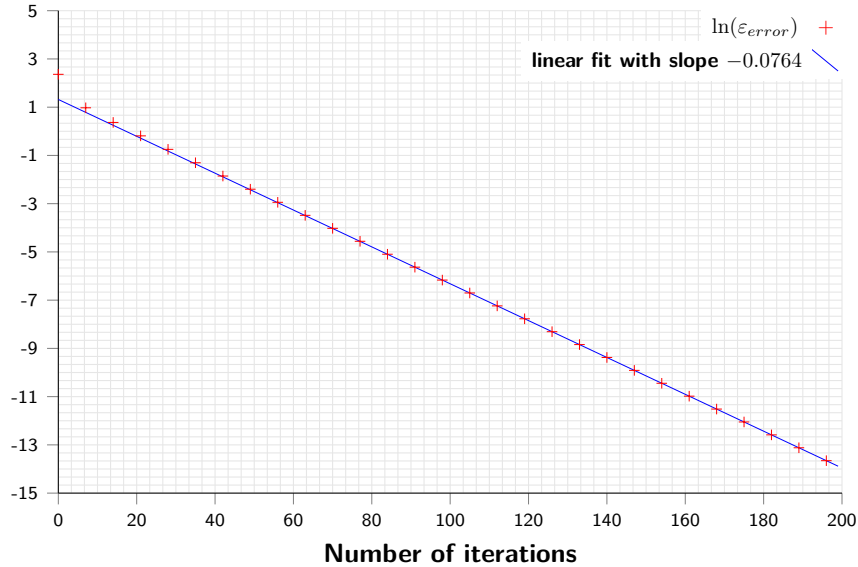
\begin{figure}[H]
    \label{MFFBSDE exemple log error in function of iterations}
\centering
\begin{tikzpicture}[
        font=\bfseries\sffamily,
    ]
    \begin{axis}[
        width=12cm,
        height=8cm,
        at={(0,0)},
        ymin=-15,
        ymax=5,
        xmin=0,
        xmax=200,
        grid=both,
        minor tick num =5,
        minor tick style={draw=none},
        minor grid style={thin,color=black!10},
        major grid style={thin,color=black!10},
        xlabel=Number of iterations,
        tick align=outside,
        axis x line*=bottom,
        axis y line*=none,
        xtick={0,20,...,200},
        extra x tick style={xticklabel style={anchor=north east}, major tick length=0pt},
        ytick={-15,-13,...,5},
        xlabel style={color=black},
        ylabel style={align=center,rotate=-90,color=black},
        x tick label style={
            /pgf/number format/assume math mode, font=\sf\scriptsize},
        y tick label style={
        /pgf/number format/assume math mode, font=\sf\scriptsize},
    ]
    \addplot[color=red,only marks,tension=0.7,each nth point=7,mark=+,very thin] table [x index=0,y index=1,col sep=space] {csv_files/error_log.csv};
    \addplot[color=blue,tension=0.7,very thin] table [x index=0,y index=1,col sep=space] {csv_files/linear_regression.csv};
    \node[text=red,align=center,fill=white,anchor=west,scale=0.8,inner sep=5pt] at (191,2.9){$\color{white}+$};
    \addplot[color=blue,tension=0.7,very thin] coordinates{(191.5,3.5)(199,2.5)};
    \node[text=black,align=center,fill=white,anchor=west,scale=0.8,inner sep=5pt] at (164,4.3){$\ln(\varepsilon_{error})$};
    \node[text=red,align=center,fill=white,anchor=west,scale=0.8,inner sep=5pt] at (191,4.3){$+$};
    \node[text=black,align=center,fill=white,anchor=west,scale=0.8,inner sep=5pt] at (118,2.9){linear fit with slope $-0.0764$};
    \end{axis}

    \end{tikzpicture}
    \caption{log-linear plot of the last iterate error in function of the number of iterations}
\medskip
\small
After a small number of iterations, the linear fit is almost perfect. This emphasizes that the last iterates converge exponentially fast even for discrete controls when the conditions of application of Theorem \ref{lemma: exponential convergence} are met.
\end{figure}
We insist that the number of iterations may appear large only because we considered a time interval of $[0,10]$ and even then, computations are quite fast. Since the functions $\theta,\eta$ can be computed very accurately in this setting let us also compare our numerical results to the solution of \eqref{eq: explicit mffbsde numerical}. For the same parameters, the functions $\eta,\theta$ are estimated numerically by doing a linear regression of $\left(Y^{i}_j\right)_{j=1:N_p}$ onto $\left(X^{i}_j\right)_{j=1:N_p}$ for fixed $1\in \{1,\cdots N_t\}$.
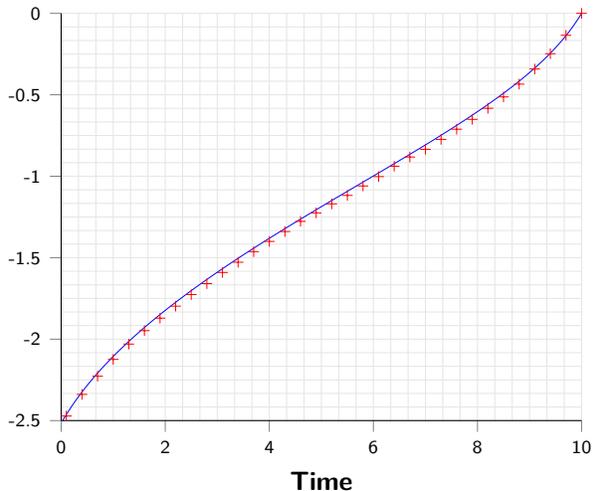
\begin{figure}[H]
\centering
\begin{minipage}{.5\textwidth}
  \centering
  \label{fig:eta}
  \captionof{figure}{  $\eta(t)$}
\begin{tikzpicture}[
        font=\bfseries\sffamily,
    ]
    \begin{axis}[
        width=7.5cm,
        height=6cm,
        at={(0,0)},
        ymin=0,
        ymax=1,
        xmin=0,
        xmax=10,
        grid=both,
        minor tick num =5,
        minor tick style={draw=none},
        minor grid style={thin,color=black!10},
        major grid style={thin,color=black!10},
        xlabel=Time,
        tick align=outside,
        axis x line*=bottom,
        axis y line*=none,
        xtick={0,2,...,10},
        extra x tick style={xticklabel style={anchor=north east}, major tick length=0pt},
        ytick={0,0.2,...,1},
        xlabel style={color=black},
        ylabel style={align=center,rotate=-90,color=black},
        x tick label style={
            /pgf/number format/assume math mode, font=\sf\scriptsize},
        y tick label style={
        /pgf/number format/assume math mode, font=\sf\scriptsize},
    ]
    \addplot[color=red,only marks,tension=0.7,each nth point=3,mark=+,very thin] table [x index=0,y index=1,col sep=space] {csv_files/observed_eta.csv};
    \addplot[color=blue,tension=0.7,very thin] table [x index=0,y index=1,col sep=space] {csv_files/true_eta.csv};
    \end{axis}

    \end{tikzpicture}
\end{minipage}%
\begin{minipage}{.5\textwidth}
  \centering
  \label{fig:theta}
   \captionof{figure}{$\theta(t)$}
\begin{tikzpicture}[
        font=\bfseries\sffamily,
    ]
    \begin{axis}[
        width=7.5cm,
        height=6cm,
        at={(0,0)},
        ymin=-2.5,
        ymax=0,
        xmin=0,
        xmax=10,
        grid=both,
        minor tick num =5,
        minor tick style={draw=none},
        minor grid style={thin,color=black!10},
        major grid style={thin,color=black!10},
        xlabel=Time,
        tick align=outside,
        axis x line*=bottom,
        axis y line*=none,
        xtick={0,2,...,10},
        extra x tick style={xticklabel style={anchor=north east}, major tick length=0pt},
        ytick={-2.5,-2,...,0},
        xlabel style={color=black},
        ylabel style={align=center,rotate=-90,color=black},
        x tick label style={
            /pgf/number format/assume math mode, font=\sf\scriptsize},
        y tick label style={
        /pgf/number format/assume math mode, font=\sf\scriptsize},
    ]
    \addplot[color=red,only marks,tension=0.7,each nth point=3,mark=+,very thin] table [x index=0,y index=1,col sep=space] {csv_files/observed_theta.csv};
    \addplot[color=blue,tension=0.7,each nth point=10,very thin] table [x index=0,y index=1,col sep=space] {csv_files/true_theta.csv};
    \end{axis}

    \end{tikzpicture}
\end{minipage}
  \medskip
\small
Plot of the functions $t\mapsto \theta(t), t\mapsto \eta(t)$ in blue and of the computed values using our algorithm with the red crosses. Since $\varepsilon_{error}$ is very small ($\sim 10^{-13}$) the difference between the true curve and the computed values does not come from the number of iterations but is rather inherent to the discretization (both in time and on the space of random paths) of the problem.
\end{figure}
\section{Perspectives}
In this article we have presented a reinterpretation of monotone FBSDEs as the solution of a monotone variational inequality in a Hilbert space. This point of view allows for a straightforward adaptation of extra-gradient methods to mean field type FBSDEs. Under sufficiently strong assumptions we present a decoupling algorithm converging exponentially fast to the solution of the problem. This is in particular the case under standard assumptions for displacement monotone MFGs. However, for FBSDEs outside of the MFG theory, the convergence rate is in general of the order of $\frac{1}{\sqrt{n}}$ for $n$ the number of iterations. A challenge for future research is to improve this convergence rate. To that end an interesting lead consists in studying a version of the algorithm with an adaptive step size. For example, although the context is very different it is shown in \cite{lavigne2023generalized} that the convergence rate of fictitious play \cite{fictitious-play-cardaliaguet} can be tremendously improved by adding a linear search for the best step size at each step. Another possibility is to find better parameterizations of the problem. In \cite{exploration-noise-delarue} the authors proved that better convergence rates can be obtained by studying FBSDEs under different probability measures using Girsanov Theorem. Perhaps such an idea can also be applied in our setting.

\section*{acknowledgements}
\begin{center}
The author is thankful to François Delarue for helpful discussions and comments on the manuscript.

\quad

Charles Meynard acknowledge the financial support of the European Research Council (ERC) under the European Union’s Horizon Europe research and innovation program (ELISA project, Grant agreement No. 101054746). Views and opinions expressed are however those of the author only and do not necessarily reflect those of the European Union or the European Research Council Executive Agency. Neither the European Union nor the granting authority can be held responsible for them
\end{center}

\bibliographystyle{plain}
\bibliography{source}
\begin{appendices}
\section{Wellposedness of a monotone FBSDE}
\begin{lemma}
    \label{lemma: wellposedness fbsde strong monotonicity on F}
Let $(F,G,g):\reels^{2d}\times \mathcal{P}_2(\reels^{2d})\to \reels^d$ be Lipschitz functions and consider the following FBSDE
\begin{equation}
    \label{eq: lemma wellposedness of fbsde}
\left\{
    \begin{array}{c}
      dX_t=-\tilde{F}(X_t,U_t)dt+\sigma dB_t, X_0=X,\\
      dU_t=\tilde{G}(X_t,U_t)dt-Z_tdB_t, U_T=g(X_T).
    \end{array}
\right.
\end{equation}
If the following monotonicity condition holds
 \begin{gather*}
\exists c_F>0, \quad \forall (X,Y,U,V)\in \mathcal{H}^{4d}, \quad \\
\langle  \tilde{g}(X)-\tilde{g}(Y),X-Y\rangle \geq 0,\\
\langle \left(\begin{array}{c}
F\\
G
\end{array}\right)
(X,Y)-\left(\begin{array}{c}
\tilde{F}\\
\tilde{G}
\end{array}\right)(Y,V),\left(\begin{array}{c}
     X-Y  \\
     U-V
\end{array}\right)\rangle \geq c_F\|\tilde{F}(X,U)-\tilde{F}(Y,V)\|^2,
\end{gather*}
then for any $T>0$, and for any initial condition $X\in \mathcal{H}$ independent of $(B_t)_{t\geq 0}$ there exists a unique solution to the FBSDE \eqref{eq: lemma wellposedness of fbsde}.
\end{lemma}
\begin{proof}
Let us present an a priori estimate for the forward backward system \eqref{eq: lemma wellposedness of fbsde}. We fix two inital conditions $X,Y\in \mathcal{H}^d$ an assume that there exist strong  solutions to \eqref{eq: lemma wellposedness of fbsde} $(X_t,U_t)_{t\in [0,T]}$ (resp. $(Y_t,V_t)_{t\in [0,T]}$) with initial condition $X$ (resp. $Y$). We now introduce the following process
\[\forall t\in [0,T], \quad I_t=(U_t-V_t)\cdot (X_t-Y_t).\]
By Ito's lemma, 
\[\forall t\in [0,T], \quad \esp{I_T}=\esp{I_t}-\int_t^T \langle \left(\begin{array}{c}
\tilde{F}\\
\tilde{G}
\end{array}\right)
(X_s,U_s)-\left(\begin{array}{c}
\tilde{F}\\
\tilde{G}
\end{array}\right)(Y_s,V_s),\left(\begin{array}{c}
     X_s-Y_s  \\
     U_s-V_s
\end{array}\right)\rangle ds.\]
Using the monotonicity of coefficients, we get that 
\[\forall t\in [0,T],\quad 0\leq \esp{I_T}\leq \esp{I_t},\]
and 
\[c_F\esp{\int_0^T|\tilde{F}(X_s,U_s)-\tilde{F}(Y_s,V_s)|^2 }ds\leq \esp{I_0}.\]
By definition of $(X_s,Y_s)_{s\in [0,T]}$, it follows naturally that 
\[\forall t\in [0,T],\quad \|X_t-Y_t\|\leq \|X-Y\|+\frac{t}{c_F}\sqrt{\esp{I_0}}.\]
Using this estimate and by a backward application of Gronwall lemma, there exists a constant $C$ depending only on $\|(G,g)\|_{Lip},c_F$  such that 
\[\forall t\in [0,T],\quad \|U_t-V_t\|^2\leq Ce^{CT}\left(\|X-Y\|^2+\esp{I_0}\right).\]
Since by Cauchy schwartz inequality
\[\forall \lambda>0, \quad \esp{I_0}\leq \frac{\lambda}{2}\|X-Y\|^2+\frac{1}{2\lambda}\|U_0-V_0\|^2,\]
choosing $\lambda= Ce^{CT}$ and evaluating for $t=0$ we get 
\[\|U_0-V_0\|^2\leq (2Ce^{CT}+C^2e^{2CT})\|X-Y\|^2.\]
This shows that there can be at most one strong solution for a given initial condition. Since this a priori estimate is valid for any initial condition and on any time interval, it follows naturally from ideas introduced in \cite{lipschitz-sol,monotone-sol-meynard} that the FBSDE \eqref{eq: lemma wellposedness of fbsde} admits a Lipschitz decoupling field $U:[0,T]\times \reels^d\times \mathcal{P}_2(\reels^d)\to \reels^d$. In particular, this allows to conclude that there indeed exists a unique strong solution to \eqref{eq: lemma wellposedness of fbsde} for any admissible initial condition $X\in \mathcal{H}^d$. 
\end{proof}
\end{appendices}
\end{document}